\documentclass[12pt]{amsart}
\usepackage{}
\usepackage{amsmath}
\usepackage{amsfonts}
\usepackage{amssymb}
\usepackage[all,cmtip]{xy}           

\usepackage{bbding}
\usepackage{txfonts}
\usepackage[shortlabels]{enumitem}
\usepackage{ifpdf}
\ifpdf
  \usepackage[colorlinks,final,backref=page,hyperindex]{hyperref}
\else
  \usepackage[colorlinks,final,backref=page,hyperindex,hypertex]{hyperref}
\fi
\usepackage{tikz}
\usepackage[active]{srcltx}

\makeatletter

\newtheorem{theorem}{Theorem}[section]
\newtheorem{prop}[theorem]{Proposition}

\newtheorem{coro}[theorem]{Corollary}
\newtheorem{prop-def}{Proposition-Definition}[section]

\theoremstyle{definition}
\newtheorem{defn}[theorem]{Definition}

\newtheorem{remark}[theorem]{Remark}

\newcommand{\nc}{\newcommand}

\nc{\delete}[1]{{}}
\nc{\mmargin}[1]{}

\nc{\mlabel}[1]{\label{#1}}  
\nc{\mcite}[1]{\cite{#1}}  
\nc{\mref}[1]{\ref{#1}}  
\nc{\mbibitem}[1]{\bibitem{#1}} 

\delete{
\nc{\mlabel}[1]{\label{#1}  
{\hfill \hspace{1cm}{\bf{{\ }\hfill(#1)}}}}
\nc{\mcite}[1]{\cite{#1}{{\bf{{\ }(#1)}}}}  
\nc{\mref}[1]{\ref{#1}{{\bf{{\ }(#1)}}}}  
\nc{\mbibitem}[1]{\bibitem[\bf #1]{#1}} 
}

\newcommand{\sha}{{\mbox{\cyr X}}}
\newcommand{\shap}{{\mbox{\cyrs X}}}
\font\cyr=wncyr10 \font\cyrs=wncyr7

\newcommand{\bcr}{{\mathfrak{B}\mathfrak{C}\mathfrak{R}}}

\newcommand{\bcd}{{\mathfrak{B}\mathfrak{C}\mathfrak{D}}}

\newcommand{\svs}{{\mathfrak{S}\mathfrak{V}\mathfrak{S}}}
\newcommand{\sca}{{\mathfrak{S}\mathfrak{C}\mathfrak{A}}}
\newcommand{\scr}{{\mathfrak{S}\mathfrak{C}\mathfrak{R}}}

\newcommand{\bk}{{\mathbf{k}}}

\nc{\vep}{\varepsilon}
\nc{\bin}[2]{ (_{\stackrel{\scs{#1}}{\scs{#2}}})}  
\nc{\binc}[2]{(\!\! \begin{array}{c} \scs{#1}\\
    \scs{#2} \end{array}\!\!)}  
\nc{\bincc}[2]{  ( {\scs{#1} \atop
    \vspace{-1cm}\scs{#2}} )}  
\nc{\bs}{\bar{S}}
\nc{\la}{\longrightarrow}
\nc{\ot}{\otimes}
\nc{\rar}{\rightarrow}
\nc{\dar}{\downarrow}
\nc{\dap}[1]{\downarrow \rlap{$\scriptstyle{#1}$}}
\nc{\defeq}{\stackrel{\rm def}{=}}
\nc{\dis}[1]{\displaystyle{#1}}
\nc{\dotcup}{\ \displaystyle{\bigcup^\bullet}\ }
\nc{\hcm}{\ \hat{,}\ }
\nc{\hts}{\hat{\otimes}}
\nc{\hcirc}{\hat{\circ}}
\nc{\lleft}{[}
\nc{\lright}{]}
\nc{\curlyl}{\left \{ \begin{array}{c} {} \\ {} \end{array}
    \right .  \!\!\!\!\!\!\!}
\nc{\curlyr}{ \!\!\!\!\!\!\!
    \left . \begin{array}{c} {} \\ {} \end{array}
    \right \} }
\nc{\longmid}{\left | \begin{array}{c} {} \\ {} \end{array}
    \right . \!\!\!\!\!\!\!}
\nc{\ora}[1]{\stackrel{#1}{\rar}}
\nc{\ola}[1]{\stackrel{#1}{\la}}
\nc{\scs}[1]{\scriptstyle{#1}} \nc{\mrm}[1]{{\rm #1}}
\nc{\dirlim}{\displaystyle{\lim_{\longrightarrow}}\,}
\nc{\invlim}{\displaystyle{\lim_{\longleftarrow}}\,}
\nc{\dislim}[1]{\displaystyle{\lim_{#1}}} \nc{\colim}{\mrm{colim}}
\nc{\mvp}{\vspace{0.3cm}} \nc{\tk}{^{(k)}} \nc{\tp}{^\prime}
\nc{\ttp}{^{\prime\prime}} \nc{\svp}{\vspace{2cm}}
\nc{\vp}{\vspace{8cm}}
\nc{\modg}[1]{\!<\!\!{#1}\!\!>}
\nc{\intg}[1]{F_C(#1)}
\nc{\lmodg}{\!<\!\!}
\nc{\rmodg}{\!\!>\!}
\nc{\cpi}{\widehat{\Pi}}
\nc{\ssha}{{\mbox{\cyrs X}}} 
\nc{\tsha}{{\mbox{\cyrt X}}}
\nc{\shpr}{\diamond}    
\nc{\labs}{\mid\!}
\nc{\rabs}{\!\mid}

\nc{\ann}{\mrm{ann}}
\nc{\Aut}{\mrm{Aut}}
\nc{\can}{\mrm{can}}
\nc{\Cont}{\mrm{Cont}}
\nc{\rchar}{\mrm{char}}
\nc{\cok}{\mrm{coker}}
\nc{\dtf}{{R-{\rm tf}}}
\nc{\dtor}{{R-{\rm tor}}}

\nc{\Div}{{\mrm Div}}
\nc{\End}{\mrm{End}}
\nc{\Ext}{\mrm{Ext}}
\nc{\Fil}{\mrm{Fil}}
\nc{\Fr}{\mrm{Fr}}
\nc{\Frob}{\mrm{Frob}}
\nc{\Gal}{\mrm{Gal}}
\nc{\GL}{\mrm{GL}}
\nc{\Hom}{\mrm{Hom}}
\nc{\hsr}{\mrm{H}}
\nc{\hpol}{\mrm{HP}}
\nc{\id}{\mrm{id}}
\nc{\im}{\mrm{im}}
\nc{\incl}{\mrm{incl}}
\nc{\length}{\mrm{length}}
\nc{\mforall}{\quad \text{for all }}
\nc{\mchar}{\rm char}
\nc{\mpart}{\mrm{part}}
\nc{\ql}{{\QQ_\ell}}
\nc{\qp}{{\QQ_p}}
\nc{\rank}{\mrm{rank}}
\nc{\rcot}{\mrm{cot}}
\nc{\rdef}{\mrm{def}}
\nc{\rdiv}{{\rm div}}
\nc{\rtf}{{\rm tf}}
\nc{\rtor}{{\rm tor}}
\nc{\res}{\mrm{res}}
\nc{\SL}{\mrm{SL}}
\nc{\Spec}{\mrm{Spec}}
\nc{\tor}{\mrm{tor}}
\nc{\Tr}{\mrm{Tr}}
\nc{\tr}{\mrm{tr}}

\nc{\bfk}{{\bf k}}
\nc{\bfone}{{\bf 1}}
\nc{\bfzero}{{\bf 0}}
\nc{\detail}{\marginpar{\bf More detail}
    \noindent{\bf Need more detail!}
    \svp}
\nc{\Diff}{\mathbf{Diff}}
\nc{\gap}{\marginpar{\bf Incomplete}\noindent{\bf Incomplete!!}
    \svp}
\nc{\FMod}{\mathbf{FMod}}
\nc{\Int}{\mathbf{Int}}
\nc{\Mon}{\mathbf{Mon}}
\nc{\remarks}{\noindent{\bf Remarks: }}
\nc{\Rep}{\mathbf{Rep}}
\nc{\Rings}{\mathbf{Rings}}
\nc{\Sets}{\mathbf{Sets}}

\nc{\BA}{{\mathbb A}}   \nc{\CC}{{\mathbb C}}
\nc{\DD}{{\mathbb D}}   \nc{\EE}{{\mathbb E}}
\nc{\FF}{{\mathbb F}}   \nc{\GG}{{\mathbb G}}
\nc{\HH}{{\mathbb H}}   \nc{\LL}{{\mathbb L}}
\nc{\NN}{{\mathbb N}}   \nc{\PP}{{\mathbb P}}
\nc{\QQ}{{\mathbb Q}}   \nc{\RR}{{\mathbb R}}
\nc{\TT}{{\mathbb T}}   \nc{\VV}{{\mathbb V}}
\nc{\ZZ}{{\mathbb Z}}   \nc{\TP}{\widetilde{P}}

\nc{\cala}{{\mathcal A}}    \nc{\calc}{{\mathcal C}}
\nc{\cald}{\mathcal{D}}     \nc{\cale}{{\mathcal E}}
\nc{\calf}{{\mathcal F}}    \nc{\calg}{{\mathcal G}}
\nc{\calh}{{\mathcal H}}    \nc{\cali}{{\mathcal I}}
\nc{\call}{{\mathcal L}}    \nc{\calm}{{\mathcal M}}
\nc{\caln}{{\mathcal N}}    \nc{\calo}{{\mathcal O}}
\nc{\calp}{{\mathcal P}}    \nc{\calr}{{\mathcal R}}
\nc{\cals}{{\mathcal S}}    \nc{\calt}{{\Omega}}
\nc{\calw}{{\mathcal W}}    \nc{\calx}{{\mathcal X}}
\nc{\CA}{\mathcal{A}}

\nc{\fraka}{{\mathfrak a}}
\nc{\frakb}{\mathfrak{b}}
\nc{\frakB}{{\frak B}} \nc{\frakm}{{\frak
m}} \nc{\frakM}{{\frak M}}
\nc{\frakp}{{\frak p}}
\nc{\frakS}{{\frak S}}
\nc{\frakA}{{\frak A}} \nc{\frakx}{{\frakx}}

\nc{\lir}[1]{\textcolor{red}{\underline{Li:}#1 }}

\begin{document}

\title[Braided Rota-Baxter, quantum quasi-shuffle and dendriform algebras]{Braided Rota-Baxter algebras, quantum quasi-shuffle algebras and braided dendriform algebras}

\author[Li Guo]{Li Guo}
\address{Department of Mathematics and Computer Science, Rutgers University, Newark, NJ 07102, USA}
\email{liguo@rutgers.edu}

\author[Yunnan Li]{Yunnan Li}
\address{School of Mathematics and Information Science, Guangzhou University, Waihuan Road West 230, Guangzhou 510006, China}
\email{ynli@gzhu.edu.cn}

\date{\today}

\begin{abstract}
Rota-Baxter algebras and the closely related dendriform algebras have important physics applications, especially to renormalization of quantum field theory. Braided structures provide effective ways of quantization such as for quantum groups. Continuing recent study relating the two structures, this paper considers Rota-Baxter algebras and dendriform algebras in the braided contexts. Applying the quantum shuffle and quantum quasi-shuffle products, we construct free objects in the categories of braided Rota-Baxter algebras and braided dendriform algebras, under the commutativity condition. We further generalize the notion of dendriform Hopf algebras to the braided context and show that quantum shuffle algebra gives a braided dendriform Hopf algebra. Enveloping braided commutative Rota-Baxter algebras of braided commutative dendriform algebras are obtained.
\end{abstract}

\subjclass[2010]{16T05,16W99,16T25,17B37}

\keywords{quantum shuffle algebra, quantum quasi-shuffle algebra, Rota-Baxter algebra, dendriform algebra, Yang-Baxter equation,
	braided Rota-Baxter algebra, braided dendriform algebra, braided dendriform Hopf algebra}

\maketitle

\tableofcontents

\allowdisplaybreaks

\section{Introduction}
This paper studies Rota-Baxter algebras and the closely related dendriform algebras in the context of braided algebras, with special attention to the free objects obtained from the quantum shuffle and quantum quasi-shuffle algebras.

\subsection{Yang-Baxter equation and braided algebras}

The Yang-Baxter equation (YBE), also called the quantum Yang-Baxter equation in some literatures to distinguish it with its classical limit, is a master equation in integral models in statistical mechanics named after C.~N. Yang and R.~J. Baxter, from their study of many-body problem~\cite{Ya} and exactly solvable lattice models~\cite{Ba1} respectively.

In the past several decades, the YBE has been studied in depth with progresses in many areas, such as $C^*$-algebras, link invariants, quantum groups, tensor categories, integrable systems and conformal field theory. For example, solutions of the YBE from quantum groups and their representations are called $R$-matrices; while in braided tensor categories, they appear as the structural ingredient, the braidings. In general, any linear operator on a space obeying the braid relation is called a Yang-Baxter operator, as the braid relation is equivalent to the YBE (without the spectral parameter). Then the space is called a \textit{braided vector space}.

Built on a braided space, the notion of a braided algebra was formulated in the early 1990's to generalize classical differential calculus in a noncommutative geometry setting~\cite{Bae1}, and was applied to study quantum analogues of linear-algebraic objects~\cite{HH}, where it was formerly named \textit{r-algebra} or \textit{Yang-Baxter algebra} respectively. Roughly speaking, a braided algebra is an algebra equipped with a Yang-Baxter operator compatible with its multiplication.

Braided algebras are generally studied as objects in braided tensor categories, such as Yetter-Drinfeld categories (see~\cite{Maj}). Concrete examples of braided algebras are commonly found in quantum group theory, in particular the positive parts of quantum groups. Also, given a representation of a quantum group, one can construct several kinds of braided algebras with their Yang-Baxter operators coming from the universal $R$-matrix of the quantum group, including quantum symmetric or exterior algebras~\cite{BZ,CTS}.

With a richer structure than braided algebras, braided Hopf algebras were widely studied. As a special class of braided Hopf algebras, Nichols algebras played a crucial role in the classification program of Hopf algebras~\cite{And, Tak}. In 1998, Rosso quantized the shuffle algebra to give an intrinsic realization of Nichols algebras~\cite{Ros}.
Later Jian and Rosso introduced quantum multi-brace algebras generalizing both braided algebras and $\mathbf{B}_\infty$-algebras~\cite{JR}. Particular interesting examples are quantum quasi-shuffle algebras~\cite{Jian,JRZ}.

Braided algebras have been applied to several other areas of mathematics and physics, such as homology~\cite{Bae,Le}, Hopf algebras~\cite{AS}, noncommutative geometry~\cite{GS} and conformal field theory~\cite{Run}.

\subsection{Rota-Baxter algebras}
Another algebraic structure of importance in physics is Rota-Baxter algebra.

A Rota-Baxter algebra is an (associative or Lie) algebra together with a linear operator satisfying a certain operator identity, called the Rota-Baxter identity. In the associative algebra context, the Rota-Baxter algebra originated from a probability study of G. Baxter~\mcite{Ba} in 1960 where he deduced important identities in fluctuation theory from the Rota-Baxter identity. Other than their theoretical significance, Rota-Baxter algebras have found broad applications in areas of mathematics and physics~\mcite{Guo,Ro}.

On the mathematical side, Rota-Baxter algebras are naturally related to dendriform algebras, tridendriform algebras and Zinbiel algebras~\mcite{Agu,EG,Lod1,LR}, and intrinsically related to quasi-symmetric functions, through their connection with quasi-shuffle product~\mcite{EG1}.

On the physics side, the Rota-Baxter algebra, together with the Hopf algebra, forms the algebraic foundation in the approach of Connes-Kreimer to renormalization of perturbative quantum field theory~\mcite{CK,EGK}. Further, the Rota-Baxter operator in the Lie algebra context is closely related to the operator form of the classical Yang-Baxter equation as the classical limit of the (quantum) Yang-Baxter equation~\mcite{Bai,STS}.

As in the case of well-known algebraic structures, free Rota-Baxter algebras (in the commutative case) were investigated in the early stage of the study. Rota~\mcite{Ro} and Cartier~\mcite{Ca} provided two constructions in late 1960s and early 1970s, with Rota's construction closely related to the Waring formula for symmetric functions and Cartier's construction built on the stuffle product preceding its formal introduction into multiple zeta values by almost three decades~\mcite{3BL}. The third construction, obtained by Guo and Keigher~\mcite{GK}, was built by mixable shuffle product, a generalization of the shuffle product and the explicit form of the quasi-shuffle product which became prominent in connection with multiple zeta values~\mcite{Ho}, but can be traced back to~\mcite{NR}.

\subsection{Braided Rota-Baxter algebras}

Motivated by the importance of quantum groups and braided algebras~\mcite{Kas,KT}, braided or quantized objects from the shuffle product, quasi-shuffle product and Rota-Baxter algebras were obtained, in~\mcite{Ros}, \mcite{Jian,JRZ} and \mcite{Jian1} respectively. See also~\mcite{Foi1} for quantization of algebra of rooted trees. Such constructions provide more examples and applications of Rota-Baxter algebras, quantum groups and braided algebras~\cite{JR}.

In~\cite{Jian1}, the notion of braided Rota-Baxter algebras was introduced. As in the case of Rota-Baxter algebras, quantum quasi-shuffle algebras were used to obtain braided Rota-Baxter algebras. Quantum quasi-shuffle algebras were also shown to be tridendriform algebras.

This paper further clarify the connections between braided algebras and quantum quasi-shuffle algebras on the one hand, and Rota-Baxter algebras and dendriform algebras on the other.

First utilizing quantum quasi-shuffle algebras and braided Rota-Baxter algebras, we obtain free commutative braided Rota-Baxter algebras in both the unitary case (Theorem~\mref{fbcr}) and the nonunitary case (Proposition~\mref{fbcr1}).

We next introduce the notion of a braided dendriform algebra and show that a strong version of braided Rota-Baxter algebra gives a braided dendriform algebra, not just a dendriform algebra. Generalizing the dendriform algebra structure on the shuffle algebra which splits the shuffle product, we derive a braided dendriform algebra structure on the quantum shuffle algebra of Rosso~\mcite{Ros}, enabling us to split the quantum shuffle product (Theorem~\mref{bcd}). Further, this structure gives the free braided commutative dendriform algebra (that is, free braided Zinbiel algebra), generalizing the classical result of Loday~\mcite{Lod}. We further equip a Hopf type structure on the quantum shuffle algebra, more precisely making it a braided dendriform Hopf algebra (Theorem~\mref{bdha}).

The left adjoint functor of the functor from Rota-Baxter algebras to dendriform algebras~\mcite{Agu} gives the enveloping (Rota-Baxter) algebra of a dendriform algebra. Taking into consideration of the braided structure, we give the enveloping algebra of a braided commutative dendriform algebra (Theorem~\mref{berba}).

\smallskip
\noindent
{\bf Notation.} In this paper, we fix a ground field $\bk$ of characteristic 0. All the objects
under discussion, including vector spaces, algebras and tensor products, are defined over $\bk$. For a vector space $V$,
we denote by $T(V)$ the tensor algebra on $V$, by $\otimes$ the tensor product within $T(V)$, and by $\underline{\otimes}$ the one between $T(V)$ and itself as in the deconcatenation coproduct. By an algebra, we mean an associative algebra which is not necessarily unital unless otherwise stated.

We denote by $\mathfrak{S}_{n}$ the symmetric group acting on the set $\{1,2,\ldots,n\}$ and by $s_{i}$, $1\leq i\leq n-1$, the standard generators of $\mathfrak{S}_{n}$ permuting $i$ and $i+1$.

\section{Braided algebras, Rota-Baxter algebras and quantum quasi-shuffle algebras}
\label{sec:freebrba}
In this section, we first recall some background on braided algebras, Rota-Baxter algebras and quantum quasi-shuffle algebras. We then show that the Rota-Baxter algebras obtained from quantum quasi-shuffle algebras are the free objects in braided commutative Rota-Baxter algebras.

\subsection{Braided algebras and Rota-Baxter algebras}
\label{ss:back}

A \textit{braiding} or \textit{Yang-Baxter operator} on a vector space $V$ is a linear
map $\sigma$ in $\mathrm{End}(V\otimes V)$ satisfying the \textit{braiding relation} on $V^{\otimes 3}$:
\begin{equation}
(\sigma\otimes
\id_{V})(\id_{V}\otimes \sigma)(\sigma\otimes
\id_{V})=(\id_{V}\otimes \sigma)(\sigma\otimes
\id_{V})(\id_{V}\otimes \sigma),
\mlabel{eq:qybe}
\end{equation}
which up to a flip is the \textit{Yang-Baxter equation} without spectral parameter. Moreover, we say that $\sigma$ is \textit{symmetric} if $\sigma^2=\id_V^{\otimes 2}$.
A \textit{braided
vector space}, denoted $(V,\sigma)$, is a vector space $V$ equipped with a braiding $\sigma$. For any $n\in \mathbb{N}$ and $1\leq i\leq
n-1$, we denote by $\sigma_i$ the operator $\id_V^{\otimes
(i-1)}\otimes \sigma\otimes \id_V^{\otimes(n-i-1)}\in
\mathrm{End}(V^{\otimes n})$.

\begin{defn}
Let $\lambda$ be a fixed element in $\bk$. A \emph{Rota-Baxter algebra of weight $\lambda$} is a pair $(R,P)$ where $R$ is
an algebra and $P$ is a linear endomorphism of $R$ satisfying
$$P(x)P(y)=P(xP(y))+P(P(x)y)+\lambda P(xy) \mforall x,y\in R.$$
The map $P$ is called a \emph{Rota-Baxter operator of weight $\lambda$}.

For two Rota-Baxter algebras $(R,P)$ and $(R',P')$ of weight $\lambda$, a map $f:R\rightarrow R'$ is called a \textit{homomorphism of Rota-Baxter algebras} if
$f$ is a homomorphism of algebras such that $fP=P'f$.
\end{defn}

\begin{defn}Let $A$ be an algebra with product $\mu$, and $\sigma$ be a braiding on $A$. We call the triple $(A,\mu,\sigma)$ a \emph{braided algebra} if it satisfies the conditions
\begin{equation}\label{ba1}
(\id_A\otimes\mu)\sigma_1\sigma_2=\sigma(\mu\otimes
\id_A),\quad
(\mu\otimes\id_A)\sigma_2\sigma_1=\sigma(\id_A\otimes\mu).
\end{equation}
Moreover, if $A$ is unital with unit $1_A$ and satisfies
\begin{equation}\label{ba2}
\sigma(a\otimes 1_A)=1_A\otimes a,\quad
\sigma(1_A\otimes a)=a\otimes1_A \quad \text{for all } a\in A, \end{equation}
then $A$ is called a \emph{unital braided algebra}. If $\mu\sigma=\mu$, then $A$ is called \textit{(braided) commutative}.

For two braided algebras $(A,\sigma)$ and $(A',\sigma')$, a map $f:A\rightarrow A'$ is called a \textit{homomorphism of braided algebras} if
$f$ is a homomorphism of algebras and $(f\otimes f)\sigma=\sigma'(f\otimes f)$.
If $A,A'$ are unital, then it also requires $f(1_A)=1_{A'}$.
\end{defn}

Following~\cite{Jian1}, we give
\begin{defn}
A triple $(R,P,\sigma)$ is called a \emph{braided (commutative) Rota-Baxter algebra of weight $\lambda$}, if $(R,\sigma)$ is a braided (commutative) algebra and $P$ is an endomorphism of $R$ such that $(R,P)$ is a Rota-Baxter algebra of weight $\lambda$ and $\sigma(P\otimes P)=(P\otimes P)\sigma$.

For two braided Rota-Baxter algebras $(R,P,\sigma)$ and $(R',P',\sigma')$ of weight $\lambda$, a map $f:R\rightarrow R'$ is called a \textit{homomorphism of braided Rota-Baxter algebras}, if $f$ is a homomorphism of Rota-Baxter algebras and of braided algebras.
\end{defn}

One can refine the structure of braided Rota-Baxter algebras as follows~\cite[\S 7]{Jian}.
\begin{defn}
A triple $(R,P,\sigma)$ is called a \emph{right (resp. left) weak braided (commutative) Rota-Baxter algebra of weight $\lambda$}, if $(R,\sigma)$ is a braided (commutative) algebra and $P$ is a Rota-Baxter operator on $R$ of weight $\lambda$ such that \begin{equation}\label{wbr}
\sigma(P\otimes \id_R)=(\id_R\otimes P)\sigma\quad\mbox{(resp. }\sigma(\id_R\otimes P)=(P\otimes \id_R)\sigma).
\end{equation}
\end{defn}
By~\cite[Proposition~3.4]{Jian1}, a right weak braided Rota-Baxter algebra which is also a left weak braided Rota-Baxter algebra is a braided Rota-Baxter algebra, which we call a
\textit{strongly braided Rota-Baxter algebra} for distinction.

For a fixed $\lambda\in\bk$, let $\bcr_\lambda$ denote the category of braided commutative unital Rota-Baxter algebras of weight $\lambda$ with their homomorphisms of braided Rota-Baxter algebras as morphisms. Also let $\scr_\lambda$ denote the full subcategoy of $\bcr_\lambda$ consisting of all strongly braided commutative unital Rota-Baxter algebras of weight $\lambda$. Then $\scr_\lambda$ is included in the category $\scr^0_\lambda$ of strongly braided commutative Rota-Baxter algebras (not necessarily unital).

\subsection{Quantum quasi-shuffle algebras and free braided commutative Rota-Baxter algebras}
\label{ss:fbrb}

Given a unital commutative algebra $A$, the \textit{free commutative Rota-Baxter algebra} on $A$, denoted by $\sha_{\bk,\lambda}(A)$ (or abbreviated as $\sha (A)$),  is constructed as follows~\cite{Guo,GK1}. As a $\bk$-module, we have
 $$\sha (A):=\bigoplus\limits_{i\geq 1}A^{\otimes i}=A\oplus (A\otimes A)\oplus (A\otimes A\otimes A)\oplus\cdots.$$
To define the multiplication $\diamond_\lambda$ on $\sha(A)$, take $\mathfrak{a}=a_0\otimes \cdots \otimes a_m\in A^{\otimes(m+1)}$ and $\mathfrak{b}=b_0\otimes \cdots \otimes b_n\in A^{\otimes(n+1)}$ with $m,n\geq0$. If $mn=0$, define

 \begin{equation}\label{shp1}
 \mathfrak{a}\diamond_\lambda \mathfrak{b}:=
 \begin{cases}
 (a_0b_0)\otimes b_1 \otimes\cdots \otimes b_n, & m=0, n>0, \\
 (a_0b_0)\otimes a_1 \otimes\cdots \otimes a_m, & m>0, n=0, \\
 a_0b_0,                                        & m=n=0. \\
 \end{cases}
 \end{equation}
 If $m>0$ and $n>0$, then $\mathfrak{a}\diamond_\lambda \mathfrak{b}$ is defined recursively on $m$ and $n$ by
 \begin{equation}\label{shp2}
\begin{split}
\mathfrak{a}\diamond_\lambda \mathfrak{b}:=&\,(a_0b_0)\otimes\big((a_1\otimes\cdots\otimes a_m)\diamond_\lambda (1_A\otimes b_1\otimes \cdots b_n)\\
&+(1_A\otimes a_1\otimes \cdots\otimes a_m)\diamond_\lambda (b_1\otimes \cdots b_n)+\lambda (a_1\otimes\cdots\otimes a_m)\diamond_\lambda (b_1\otimes \cdots b_n)\big).
\end{split}
 \end{equation}
 The Rota-Baxter operator $P_{\shap(A)}$ on $\sha(A)$ of weight $\lambda$ is defined by
 \begin{equation}
 P_{\shap(A)}(x_0\otimes \cdots \otimes x_n)=1_A \otimes x_0\otimes\cdots\otimes x_n.
 \label{eq:rbo}
 \end{equation}

We have the tensor product of algebras $\sha (A)=A\otimes\sha^+(A)$, where $\sha^+(A):=\bigoplus_{i\geq 0} A^{\ot i}$ is equipped with the mixable shuffle product~\cite{GK1} which is identified with the quasi-shuffle product~\cite{EG1,Guo}. The algebra $\sha^+(A)$ (resp. $\sha(A)$) is called the (resp. \textit{augmented}) \textit{mixable shuffle} Rota-Baxter algebra of weight $\lambda$, and $\sha^+(A)$ can be naturally embedded as the Rota-Baxter subalgebra  $1_A\otimes\sha^+(A)$ of $\sha(A)$.

Next we consider the braided version of $\sha(A)$ as in \cite{Jian1}. For this purpose, we introduce more notations.

For any $w\in \mathfrak{S}_{n}$, we
denote by $T^\sigma_w$ the corresponding lift of $w$ in the braid
group $B_n$, defined as follows: if $w=s_{i_1}\cdots s_{i_l}$ is
any reduced expression of $w$, then $T^\sigma_w:=\sigma_{i_1}\cdots
\sigma_{i_l}$. This definition is well defined (see, e.g., Theorem
4.12 in \cite{KT}). When $\sigma$ is the usual flip $\tau$, it reduces to the permutation action of $\mathfrak{S}_{n}$ on $V^{\otimes n}$, namely,
\[T^\tau_w(v_1\otimes\cdots\otimes v_n)=v_{w^{-1}(1)}\otimes\cdots\otimes v_{w^{-1}(n)}.\]

We also define $\beta:T(V)\underline{\otimes} T(V)\rightarrow
T(V)\underline{\otimes} T(V)$ by requiring that, for $i,j\geq 1$,  the restriction $\beta_{ij}$ of
$\beta$ to $V^{\otimes i}\underline{\otimes} V^{\otimes j}$ is $T^\sigma_{\chi_{ij}}$ ,where
\[\chi_{ij}=\left(\begin{array}{cccccccc}
1&2&\cdots&i&i+1&i+2&\cdots & i+j\\
j+1&j+2&\cdots&j+i&1& 2 &\cdots & j
\end{array}\right)\in \mathfrak{S}_{i+j}.\]
For convenience, we
interpret $\beta_{0i}$ and $\beta_{i0}$ as the identity map on
$V^{\otimes i}$.  It is easy to verify the following equalities of $\beta$ on $T(V)$:
\begin{equation}\label{beta}
\beta_{m+n,k}=(\beta_{mk}\otimes\mbox{id}_V^{\otimes n})(\mbox{id}_V^{\otimes m}\otimes\beta_{nk}),\,
\beta_{m,n+k}=(\mbox{id}_V^{\otimes n}\otimes\beta_{mk})(\beta_{mn}\otimes\mbox{id}_V^{\otimes k}),\,m,n,k\geq0,
\end{equation}
which will be used throughout the paper. For the convenience of constructions below, we also introduce \textit{shuffles of permutations}, as the set of shuffle representatives of $\mathfrak{S}_{i+j}/(\mathfrak{S}_i\times\mathfrak{S}_j)$,
$$\mathfrak{S}_{i,j}:=\left\{w\in\mathfrak{S}_{i+j}\,\bigg|\,{w(1)<\cdots<w(i)\atop w(i+1)<\cdots<w(i+j)}\right\}.$$
In particular, let $\mathfrak{S}_{i,0}=\mathfrak{S}_{0,i}=\{(1)\}$ by convention.

Given a braided commutative unital $\bk$-algebra $A$ with product $\mu$ and braiding $\sigma$, first we take $$\sha^+_\sigma(A):=\sha^+_{\sigma,\bk,\lambda}(A) :=\bigoplus_{i\geq0}A^{\otimes i},$$
equipped with the \textit{quantum quasi-shuffle algebra} introduced in \cite{JRZ}, with the product $*_{\sigma,\lambda}$ defined recursively as a sum $*_\sigma=\sum_{i,j\geq 0} *_{\sigma(i,j)}$. Here the maps $$*_{\sigma(i,j)}:V^{\otimes i}\underline{\otimes}V^{\otimes j}\longrightarrow \sha^+_\sigma(A)$$ are recursively defined by

\begin{equation}
\begin{split}
1_\bk*_{\sigma(0,0),\lambda}1_\bk&:=1_\bk, \\
1_\bk*_{\sigma(0,j),\lambda}\mathfrak{a}& :=\mathfrak{a}*_{\sigma(i,0),\lambda}1_\bk=\mathfrak{a},\\
\mathfrak{a}*_{\sigma(i,j),\lambda}\mathfrak{b}&:=
(\id_A\otimes*_{\sigma(i-1,j),\lambda})\left(\id_A^{\otimes(i+j)}
+\id_A\otimes\beta_{i-1,1}\otimes\id_A^{\otimes(j-1)}\right) (\mathfrak{a}\otimes\mathfrak{b})\\
&+\lambda(\mu\otimes*_{\sigma(i-1,j-1),\lambda}) \left(\id_A\otimes\beta_{i-1,1}\otimes\id_A^{\otimes(j-1)}\right)
(\mathfrak{a}\otimes\mathfrak{b}),
\end{split}
\label{eq:qqs}
\end{equation}
where $\mathfrak{a}\in A^{\otimes i},\,\mathfrak{b}\in A^{\otimes j}$ with $i,j>0$.

Then let
\[\sha_{\sigma,\bk,\lambda}(A) :=A\otimes\sha^+_{\sigma,\bk,\lambda}(A) =\bigoplus_{i\geq1}A^{\otimes i},\]
abbreviated as $\sha_\sigma(A)=A\otimes\sha^+_\sigma(A)$. So its underlying space is the same as $\sha(A)$, but with the modified multiplication $\diamond_{\sigma,\lambda}$ from  $\diamond_\lambda$ defined by
\begin{equation}\label{qmsh}
\mathfrak{a}\diamond_{\sigma,\lambda}\mathfrak{b}
=(\mu\otimes*_{\sigma(m,n),\lambda})
(\id_A\otimes\beta_{m,1}\otimes\id_A^{\otimes n})
(\mathfrak{a}\otimes\mathfrak{b}),
\end{equation}
for $\mathfrak{a}\in A^{\otimes(m+1)},\,\mathfrak{b}\in A^{\otimes(n+1)}$ with $m,n\geq0$. Namely, $(\sha_\sigma(A),\diamond_{\sigma,\lambda},\beta)$ is the braided tensor product algebra of $(A,\mu,\sigma)$ and  $(\sha^+_\sigma(A),*_{\sigma,\lambda},\beta)$.

We note that when $(A,\mu,\sigma)$ is commutative, additional conditions are needed for $(\sha_\sigma(A),\diamond_{\sigma,\lambda},\beta)$ to be commutative. In fact,  $(\sha_\sigma(A),\diamond_{\sigma,\lambda},\beta)$ is commutative and $\beta^2=\id_{\shap_\sigma(A)}^{\otimes 2}$ if and only if $\sigma^2=\id_A^{\otimes 2}$. See~\cite[Lemma 3]{Bae} and \cite[Theorem 15]{JRZ}.

There is the canonical embedding
\[j_A:A\rightarrow\sha_\sigma(A),\,a\mapsto a\otimes 1_\bk,\]
as braided algebras. Also, take the linear operator $P_{\shap(A)}$ on $\sha_\sigma(A)$ from \eqref{eq:rbo} but  denote by $P_{\shap_\sigma(A)}$ for distinction. Then we have the following result.
\begin{prop}\label{wrbr}
For any braided unital algebra $(A,\mu,\sigma)$, the quadruple \[(\sha_\sigma(A),\diamond_{\sigma,\lambda},P_{\shap_\sigma(A)},\beta)\]
is a strongly braided unital Rota-Baxter algebra of weight $\lambda$.
\end{prop}
\begin{proof}
In fact by \cite[Theorem 2.4, Example 3.2]{Jian1},  $(\sha_\sigma(A),\diamond_{\sigma,\lambda},P_{\shap_\sigma(A)},\beta)$ is a braided unital Rota-Baxter algebra. So we further check conditions in \eqref{wbr} as follows.
For $\mathfrak{a}\in A^{\otimes(m+1)}$ and $\mathfrak{b}\in A^{\otimes(n+1)}$ with $m,n\geq0$,
\begin{align*}
&\begin{split}
\beta&(P_{\shap_\sigma(A)}\otimes\id_{\shap_\sigma(A)})(\mathfrak{a}\otimes\mathfrak{b})
=\beta_{m+2,n+1}((1_A\otimes\mathfrak{a})\otimes\mathfrak{b})\\
&=(\beta_{1,n+1}\otimes\id_A^{\otimes
(m+1)})(\id_A\otimes\beta_{m+1,n+1})(1_A\otimes\mathfrak{a}\otimes\mathfrak{b})\\
&=(\id_{\shap_\sigma(A)}\otimes P_{\shap_\sigma(A)})\beta(\mathfrak{a}\otimes\mathfrak{b}),
\end{split}\\
&\begin{split}
\beta&(\id_{\shap_\sigma(A)}\otimes P_{\shap_\sigma(A)})(\mathfrak{a}\otimes\mathfrak{b})
=\beta_{m+1,n+2}(\mathfrak{a}\otimes(1_A\otimes\mathfrak{b}))\\
&=(\id_A\otimes\beta_{m+1,n+1})
(\beta_{m+1,1}\otimes\id_A^{\otimes(n+1)})
(\mathfrak{a}\otimes1_A\otimes\mathfrak{b})\\
&=(\id_A\otimes\beta_{m+1,n+1})
(1_A\otimes\mathfrak{a}\otimes\mathfrak{b})\\
&=(P_{\shap_\sigma(A)}\otimes \id_{\shap_\sigma(A)})\beta(\mathfrak{a}\otimes\mathfrak{b}),
\end{split}
\end{align*}
where we have used equality \eqref{beta} and condition \eqref{ba2}.
\end{proof}

Now we are in the position to give our first main theorem, as a braided generalization of the mixable shuffle (that is, quasi-shuffle) construction of free commutative Rota-Baxter algebras~\cite[Theorem 4.1]{GK1} recalled at the beginning of this subsection.

\begin{theorem}\label{fbcr}
For any braided commutative unital algebra $(A,\mu,\sigma)$ with $\sigma^2=\id_A^{\otimes 2}$, the quadruple $(\sha_\sigma(A),\diamond_{\sigma,\lambda},P_{\shap_\sigma(A)},\beta)$ is the free object in $\scr_\lambda$. More precisely, for any $(R,P_R,\tau)$ in $\scr_\lambda$ with braided algebra homomorphism $\varphi:A\rightarrow R$, there exists a unique morphism $\bar{\varphi}:\sha_\sigma(A)\rightarrow R$ in $\scr_\lambda$ such that the following commutative diagram holds,
\[\xymatrix@=2em{A\ar@{->}[d]_-{\varphi}\ar@{->}[r]^-{j_A}& {\sha_\sigma(A)}\ar@{.>}[dl]^-{\bar{\varphi}}\\
R&}\]
\end{theorem}
\begin{proof}
By Proposition~\ref{wrbr}, $(\sha_\sigma(A),\diamond_{\sigma,\lambda},P_{\shap_\sigma(A)},\beta)$ is a strongly braided Rota-Baxter algebra of weight $\lambda$. We also check that it is commutative when $(A,\mu,\sigma)$ is commutative and $\sigma^2=\id_A^{\otimes 2}$. In fact, for $\mathfrak{a}\in A^{\otimes(m+1)}$ and $\mathfrak{b}\in A^{\otimes(n+1)},\,m,n\geq0$,
\[\begin{split}
\diamond_{\sigma,\lambda}&\beta(\mathfrak{a}\otimes\mathfrak{b})=(\mu\otimes*_{\sigma(n,m),\lambda})(\id_A\otimes\beta_{n,1}\otimes\id_A^{\otimes m})\beta_{m+1,n+1}(\mathfrak{a}\otimes\mathfrak{b})\\
&=(\mu\otimes*_{\sigma(n,m),\lambda})((\id_A\otimes\beta_{n,1})\beta_{1,n+1}\otimes\id_A^{\otimes m})(\id_A\otimes \beta_{m,n+1})(\mathfrak{a}\otimes\mathfrak{b})\\
&=(\mu\otimes*_{\sigma(n,m),\lambda})(\beta_{1,1}\otimes\id_A^{\otimes(m+n)})(\id_A\otimes \beta_{m,n+1})(\mathfrak{a}\otimes\mathfrak{b})\\
&=(\mu\otimes*_{\sigma(n,m),\lambda})(\beta_{1,1}\otimes\beta_{m,n})(\id_A\otimes\beta_{m,1}\otimes\id_A^{\otimes n})(\mathfrak{a}\otimes\mathfrak{b})\\
&=(\mu\otimes*_{\sigma(m,n),\lambda})(\id_A\otimes\beta_{m,1}\otimes\id_A^{\otimes n})(\mathfrak{a}\otimes\mathfrak{b})=\mathfrak{a}\diamond_{\sigma,\lambda}\mathfrak{b},
\end{split}\]
where we use \eqref{beta} and \eqref{qmsh}, also $\beta^2=\id_{\shap_\sigma(A)}^{\otimes 2}$ and $*_{\sigma,\lambda}\beta=*_{\sigma,\lambda}$ when $\sigma^2=\id_A^{\otimes 2}$ as in \cite[Lemma 14, Theorem 15]{JRZ}.

To verify the universal property of $\sha_\sigma(A)$, fix a strongly braided commutative Rota-Baxter algebra $(R,P_R,\tau)$ and let  $\varphi:A\rightarrow R$ be a braided algebra homomorphism. Define the linear map $\bar{\varphi}:\sha_\sigma(A)\rightarrow R$ as follows.
For $\mathfrak{a}=a_0\otimes \cdots \otimes a_m\in A^{\otimes(m+1)},\,m\geq0$, let
\begin{equation}\label{varp}
\bar{\varphi}(\mathfrak{a}):=\varphi(a_0)P_R(\varphi(a_1)P_R(\varphi(a_2)\cdots)).
\end{equation}
In fact, one can derive it by induction on $m$. For $m=0$, since we need $\bar{\varphi}j_A=\varphi$, it requires $\bar{\varphi}(a)=a$ for $a\in A$.
For any $\mathfrak{a}=a_0\otimes\mathfrak{a}'\in A^{\otimes(m+1)},\,m>0$, we have $\mathfrak{a}=a_0\diamond_{\sigma,\lambda}P_{\shap_\sigma(A)}(\mathfrak{a}')$ and thus
\[\bar{\varphi}(\mathfrak{a})=\bar{\varphi}(a_0)\bar{\varphi}(P_{\shap_\sigma(A)}(\mathfrak{a}'))=\varphi(a_0)P_R(\bar{\varphi}(\mathfrak{a}')).\]
That also implies the uniqueness of $\bar{\varphi}$.

Next we check that such $\bar{\varphi}:\sha_\sigma(A)\rightarrow R$ just defined is a
homomorphism of braided Rota-Baxter algebras. First by the construction of $\bar{\varphi}$, we have
\[\bar{\varphi}(P_{\shap_\sigma(A)}(\mathfrak{a}))=
\bar{\varphi}(1_A\otimes\mathfrak{a})
=\varphi(1_A)P_R(\bar{\varphi}(\mathfrak{a}))
=P_R(\bar{\varphi}(\mathfrak{a})).\]
Thus $\bar{\varphi}P_{\shap_\sigma(A)}=P_R\bar{\varphi}$.

For the commutativity of $\bar{\varphi}$ with respect to the braidings, let $\mathfrak{a}=a_0\otimes\mathfrak{a}'\in A^{\otimes(m+1)},\mathfrak{b}=b_0\otimes\mathfrak{b}'\in A^{\otimes(n+1)},\,m,n\geq0$, and we prove the commutativity by induction on $m+n$. When $m=n=0$, i.e. $a,b\in A$, we have
\[(\bar{\varphi}\otimes\bar{\varphi})\beta(a\otimes b)=(\varphi\otimes\varphi)\sigma(a\otimes b)
=\sigma(\varphi\otimes\varphi)(a\otimes b)=\tau(\bar{\varphi}\otimes\bar{\varphi})(a\otimes b).\]
For $m,n>0$,  we get
\begin{align*}
\tau(\bar{\varphi}&\otimes\bar{\varphi})(\mathfrak{a}\otimes\mathfrak{b})=\tau(\varphi(a_0)P_R(\bar{\varphi}(\mathfrak{a}'))\otimes\varphi(b_0)P_R(\bar{\varphi}(\mathfrak{b}')))\\
&=(\id_R\otimes\mu_R)\tau_1\tau_2(\varphi(a_0)\otimes P_R(\bar{\varphi}(\mathfrak{a}'))\otimes\varphi(b_0)P_R(\bar{\varphi}(\mathfrak{b}')))\\
&=(\id_R\otimes\mu_R)\tau_1(\id_R\otimes\mu_R\otimes\id_R)\tau_3\tau_2(\varphi(a_0)\otimes P_R(\bar{\varphi}(\mathfrak{a}'))\otimes\varphi(b_0)\otimes P_R(\bar{\varphi}(\mathfrak{b}')))\\
&=(\mu_R\otimes\mu_R)\tau_2\tau_1\tau_3\tau_2(\varphi(a_0)\otimes P_R(\bar{\varphi}(\mathfrak{a}'))\otimes\varphi(b_0)\otimes P_R(\bar{\varphi}(\mathfrak{b}')))\\
&=(\mu_R\otimes\mu_R)(\varphi\otimes P_R\bar{\varphi}\otimes\varphi\otimes P_R\bar{\varphi})\beta_2\beta_1\beta_3\beta_2
(a_0\otimes\mathfrak{a}'\otimes b_0\otimes\mathfrak{b}')\\
&=(\bar{\varphi}\otimes\bar{\varphi})\beta(\mathfrak{a}\otimes\mathfrak{b}),
\end{align*}
where we have used \eqref{ba1} for the second to fourth equalities, then \eqref{wbr} and the induction hypothesis for the fifth one. If exactly one of $m,n$ is 0, the situation is similar to check. Hence, we have proved that $\tau(\bar{\varphi}\otimes\bar{\varphi})=(\bar{\varphi}\otimes\bar{\varphi})\beta$.

To finish the proof, we need to prove the multiplicity of $\bar{\varphi}$. Given $\mathfrak{a}=a_0\otimes\mathfrak{a}'\in A^{\otimes(m+1)},\mathfrak{b}=b_0\otimes\mathfrak{b}'\in A^{\otimes(n+1)},\,m,n\geq0$,
we prove the multiplicity by induction on $m+n$. When $m=n=0$, i.e. $a,b\in A$, we have \[\bar{\varphi}(a\diamond_{\sigma,\lambda}b)=\bar{\varphi}(ab)
=\varphi(ab)=\varphi(a)\varphi(b)=\bar{\varphi}(a)\bar{\varphi}(b).\]
If exactly one of $m,n$ is 0, then
\[\bar{\varphi}(a\diamond_{\sigma,\lambda}\mathfrak{b})=
\bar{\varphi}(ab_0\otimes\mathfrak{b}')=\varphi(ab_0)P_R(\bar{\varphi}(\mathfrak{b}'))
=\varphi(a)\varphi(b_0)P_R(\bar{\varphi}(\mathfrak{b}'))
=\bar{\varphi}(a)\bar{\varphi}(\mathfrak{b}),\]
when $m=0,n>0$, and
\[\begin{split}
\bar{\varphi}&(\mathfrak{a}\diamond_{\sigma,\lambda}b)=
\mu_R(\varphi\otimes P_R\bar{\varphi})(\mu\otimes\id_A^{\otimes m})(\id_A\otimes\beta_{m,1})(\mathfrak{a}\otimes b)\\
&=\mu_R(\mu_R\otimes P_R)
(\varphi\otimes\varphi\otimes\bar{\varphi})
(\id_A\otimes\beta_{m,1})(\mathfrak{a}\otimes b)\\
&=\mu_R(\mu_R\otimes P_R)
(\id_R\otimes\tau)(\varphi\otimes\bar{\varphi}\otimes\varphi)
(\mathfrak{a}\otimes b)\\
&=\mu_R(\id_R\otimes\mu_R\tau)(\id_R\otimes P_R\otimes\id_R)(\varphi\otimes\bar{\varphi}\otimes\varphi)
(\mathfrak{a}\otimes b)\\
&=\mu_R(\id_R\otimes\mu_R)(\varphi\otimes P_R\bar{\varphi}\otimes\varphi)
(\mathfrak{a}\otimes b)=\bar{\varphi}(\mathfrak{a})\bar{\varphi}(b)
\end{split}\]
when $m>0,n=0$. Here we have used the commutativity of $\bar{\varphi}$ with respect to the braidings for the third equality, the associativity of $\mu_R$ together with condition \eqref{wbr} for the fourth one, and the commutativity of $\mu_R$ for the fifth one.

For $m,n>0$, first note that
\[\mathfrak{a}\diamond_{\sigma,\lambda}\mathfrak{b}=\diamond_{\sigma(1,m+n+1),\lambda}(\id_A\otimes\diamond_{\sigma(m+1,n+1),\lambda})(\mu\otimes P_{A^{\otimes m}}\otimes P_{A^{\otimes n}})(\id_A\otimes\beta_{m,1}\otimes\id_A^{\otimes n})(\mathfrak{a}\otimes\mathfrak{b})\]
by definition, with $P_{A^{\otimes i}}$ denoting the restriction of $P_{\shap_\sigma(A)}$ to the subspace $A^{\otimes i},\,i>0$.   Therefore,
\begin{align*}
\bar{\varphi}(\mathfrak{a}\diamond_{\sigma,\lambda}\mathfrak{b})&
=\bar{\varphi}\diamond_{\sigma(1,m+n+1),\lambda}(\id_A\otimes\diamond_{\sigma(m+1,n+1),\lambda})(\mu\otimes P_{A^{\otimes m}}\otimes P_{A^{\otimes n}})(\id_A\otimes\beta_{m,1}\otimes\id_A^{\otimes n})(\mathfrak{a}\otimes\mathfrak{b})\\
&=\mu_R(\varphi\otimes\bar{\varphi})(\id_A\otimes\diamond_{\sigma(m+1,n+1),\lambda})
(\mu\otimes P_{A^{\otimes m}}\otimes P_{A^{\otimes n}})(\id_A\otimes\beta_{m,1}\otimes\id_A^{\otimes n})(\mathfrak{a}\otimes\mathfrak{b})\\
&=\mu_R(\mu_R\otimes\bar{\varphi})(\varphi\otimes\varphi\otimes P_{\shap_\sigma(A)}\diamond_{\sigma,\lambda})
(\id_A^{\otimes 2}\otimes(P_{A^{\otimes m}}\otimes\id_A^{\otimes n}+\id_A^{\otimes m}\otimes P_{A^{\otimes n}}+\lambda\id_A^{\otimes m}\otimes\id_A^{\otimes n}))\\
&\quad(\id_A\otimes\beta_{m,1}\otimes\id_A^{\otimes n})(\mathfrak{a}\otimes\mathfrak{b})\\
&=\mu_R(\mu_R\otimes P_R)(\varphi\otimes\varphi\otimes\bar{\varphi}\diamond_{\sigma,\lambda})
(\id_A^{\otimes 2}\otimes(P_{A^{\otimes m}}\otimes\id_A^{\otimes n}+\id_A^{\otimes m}\otimes P_{A^{\otimes n}}+\lambda\id_A^{\otimes m}\otimes\id_A^{\otimes n}))\\
&\quad(\id_A\otimes\beta_{m,1}\otimes\id_A^{\otimes n})(\mathfrak{a}\otimes\mathfrak{b})\\
&=\mu_R(\mu_R\otimes P_R\mu_R)(\varphi\otimes\varphi\otimes\bar{\varphi}\otimes\bar{\varphi})
(\id_A^{\otimes 2}\otimes(P_{A^{\otimes m}}\otimes\id_A^{\otimes n}+\id_A^{\otimes m}\otimes P_{A^{\otimes n}}+\lambda\id_A^{\otimes m}\otimes\id_A^{\otimes n}))\\
&\quad(\id_A\otimes\beta_{m,1}\otimes\id_A^{\otimes n})(\mathfrak{a}\otimes\mathfrak{b})\\
&=\mu_R(\mu_R\otimes P_R\mu_R)
(\id_R^{\otimes 2}\otimes(P_R\otimes\id_R+\id_R\otimes P_R+\lambda\id_R^{\otimes 2}))(\varphi\otimes\varphi\otimes\bar{\varphi}\otimes\bar{\varphi})\\
&\quad(\id_A\otimes\beta_{m,1}\otimes\id_A^{\otimes n})(\mathfrak{a}\otimes\mathfrak{b})\\
&=\mu_R(\mu_R\otimes\mu_R)(\varphi\otimes\varphi\otimes P_R\bar{\varphi}\otimes P_R\bar{\varphi})(\id_A\otimes\beta_{m,1}\otimes\id_A^{\otimes n})(\mathfrak{a}\otimes\mathfrak{b})\\
&=\mu_R(\mu_R\otimes\mu_R)(\id_R\otimes\tau\otimes\id_R)(\varphi\otimes P_R\bar{\varphi}\otimes\varphi\otimes P_R\bar{\varphi})(\mathfrak{a}\otimes\mathfrak{b})\\
&=\mu_R(\id_R\otimes\mu_R)(\id_R\otimes\mu_R\tau\otimes\id_R)(\varphi\otimes P_R\bar{\varphi}\otimes\varphi\otimes P_R\bar{\varphi})(\mathfrak{a}\otimes\mathfrak{b})\\
&=\mu_R(\mu_R\otimes\mu_R)(\varphi\otimes P_R\bar{\varphi}\otimes\varphi\otimes P_R\bar{\varphi})(\mathfrak{a}\otimes\mathfrak{b})=\bar{\varphi}(\mathfrak{a})\bar{\varphi}(\mathfrak{b}),
\end{align*}
where we have used the multiplicity of the previous case for the second equality, the multiplicity of $\varphi$ and the Rota-Baxter condition for the third one, the induction hypothesis for the fifth one, and the commutativity of $\mu_R$ for the last one.
\end{proof}

Next we consider the free object over a braided commutative algebra $(A,\mu,\sigma)$ (not nesessarily unital) in the category of braided communitative Rota-Baxter algebras. First let $\bar{A}=A\oplus\bk$ be the augmented unital algebra of $A$ with the natural embedding $\iota_A:A\rightarrow\bar{A},\,a\mapsto(a,0)$. Then $\bar{A}$ is an induced braided unital algebra with the braiding
\begin{align*}
&\sigma((a,x)\otimes(b,y))=(\iota_A\otimes\iota_A)\sigma(a\otimes b)+(0,y)\otimes(a,x)+(b,0)\otimes(0,x),
\end{align*}
for $a,b\in A$ and $x,y\in\bk$. Referring to the construction in \cite{GK} , we define
\[\sha_\sigma^0(A):=\bigoplus_{i\geq0}\bar{A}^{\otimes i}\otimes A.\]
It is easy to check that $\sha_\sigma^0(A)$ is a braided commutative Rota-Baxter subalgebra of $\sha_\sigma(\bar{A})$, with $j_A:A\rightarrow \sha_\sigma^0(A)$ as a restriction of the original embedding $j_{\bar{A}}$ to $A$.

As a braided version of \cite[Proposition~2.6]{GK}, we have

\begin{prop}\label{fbcr1}
For any braided commutative algebra $(A,\mu,\sigma)$ with $\sigma^2=\id_A^{\otimes 2}$, the quadruple $(\sha_\sigma^0(A),\diamond_{\sigma,\lambda},P_{\shap_\sigma(\bar{A})},\beta)$ satisfies a universal property similar to Theorem \ref{fbcr}: for any $(R,P_R,\tau)\in\scr^0_\lambda$ with braided algebra homomorphism $\varphi:A\rightarrow R$, there exists a unique morphism $\bar{\varphi}:\sha_\sigma^0(A)\rightarrow R$ in $\scr^0_\lambda$ such that
$\varphi=\bar{\varphi} j_A$.
\end{prop}
\begin{proof}
We only need to check the existence and uniqueness of the morphism $\bar{\varphi}$.
It can be constructed as in \eqref{varp}. For any $\mathfrak{a}=a_0\otimes a_1\otimes\cdots\otimes a_m\in\bar{A}^{\otimes m}\otimes A,\,m\geq0$, we define $\bar{\varphi}(\mathfrak{a})$ by induction on $m$. If $m=0$, then $\mathfrak{a}\in A$, and we need $\bar{\varphi}(\mathfrak{a})=\varphi(\mathfrak{a})$. For $m>0$, denote $a=(a'_0,x)\in \bar{A}$ and define
	 \[\bar{\varphi}(\mathfrak{a})=(\varphi(a'_0)+x\mbox{id}_R)P_R(\bar{\varphi}(a_1\otimes\cdots\otimes a_m)).\]
	In fact, as $\bar{\varphi}$ is required to be a homomorphism of braided Rota-Baxter algebras,
	there is
	\[\begin{split}
	\bar{\varphi}(\mathfrak{a})
&=\bar{\varphi}((a'_0,0)\diamond_{\sigma,\lambda}((0,1_\bk)\otimes a_1\otimes\cdots\otimes a_m)+(0,x)\otimes a_1\otimes\cdots\otimes a_m)\\
&=(\bar{\varphi}((a'_0,0))+x\mbox{id}_R)\bar{\varphi}(P_{\shap_\sigma(\bar{A})}(a_1\otimes\cdots\otimes a_m))=(\varphi(a'_0)+x\mbox{id}_R)P_R(\bar{\varphi}(a_1\otimes\cdots\otimes a_m)).
	\end{split}\]
	Thus this is the only way to define the desired homomorphism $\bar{\varphi}$. On the other hand, the condition that $\bar{\varphi}$ is a homomorphism in $\scr^0_\lambda$ such that $\varphi=\bar{\varphi}\circ j_A$ is guaranteed by its augmented unital case in Theorem \ref{fbcr}.
\end{proof}

It is interesting to further consider the realization of the free object in $\bcr_\lambda$ consisting of braided commutative unital Rota-Baxter algebras of weight $\lambda$.
	
Let $\sca$ be the category of braided commutative unital algebras with symmetric braidings, and $F:\scr_\lambda\rightarrow\sca$ be the forgetful functor ignoring the Rota-Baxter operators. Then the functor $\sha_{\sigma,\lambda}:\sca\rightarrow\scr_\lambda$ is just the left adjoint functor to $F$ by Theorem \ref{fbcr}, that is, \[\mathrm{Hom}_\sca(A,F(R))\cong\mathrm{Hom}_{\scr_\lambda}(\sha_{\sigma,\lambda}(A),R).\]

Analogously, let $\sca^0$ be the category of braided commutative algebras (not necessarily unital) with symmetric braidings. Then the corresponding forgetful functor $F^0:\scr^0_\lambda\rightarrow\sca^0$ has its left adjoint functor $\sha_{\sigma,\lambda}^0:\sca^0\rightarrow\scr^0_\lambda$ by
Proposition~\ref{fbcr1}, that is,	 \[\mathrm{Hom}_{\sca^0}(A,F^0(R))\cong\mathrm{Hom}_{\scr^0_\lambda}(\sha_{\sigma,\lambda}^0(A),R).\]

Note that $\sca$ is not a full subcategory of $\sca^0$. That is because the homomorphisms in $\sca$ preserve units, whereas those in $\sca^0$ are free of this restriction. Hence, for free objects $\sha_{\sigma,\lambda}(A)$ and $\sha_{\sigma,\lambda}^0(A)$ over a braided commutative unital algebra $(A,\mu,\sigma)$, the second one is larger than the first one.

\section{Quantum shuffle algebras and free braided commutative dendriform algebras}
In this section we introduce the notion of a braided dendriform algebra and show that, under the symmetric condition of the braiding, quantum shuffle algebras gives free braided dendriform algebras.

We first recall the well-known notion of a dendriform algebra of Loday~\cite{Lod1}.
\begin{defn}
A triple $(D,\prec,\succ)$ is
called a \emph{dendriform algebra} if $D$ is a $\bk$-module with two binary operations $\prec,\succ$ satisfying the relations
\begin{align}
&\label{prec}(x\prec y)\prec z=x\prec(y\prec z+y\succ z),\\
&\label{ps}(x\succ y)\prec z= x\succ (y \prec z),\\
&\label{succ}x\succ (y\succ z)=(x\prec y+x\succ y)\succ z,
\end{align}
for $x,y,z\in D$. Moreover, a dendriform algebra $D$ is called \textit{commutative} if $x\prec y=y\prec x$ for any $x,y\in D$, when it is also called a \textit{Zinbiel algebra}.

For two dendriform algebras $(D,\prec,\succ)$ and $(D',\prec',\succ')$, a map $f:D\rightarrow D'$ is called a \textit{homomorphism of dendriform algebras} if
$f$ is a linear homomorphism such that $f\prec=\prec'(f\otimes f)$ and $f\succ=\succ'(f\otimes f)$.
\end{defn}

Next we introduce the braided analogue of dendriform algebras.
\begin{defn}
	A triple $(D,\prec,\succ,\sigma)$ is called a \textit{braided dendriform algebra} if $(D,\sigma)$ is a braided vector space and $(D,\prec,\succ)$ is a dendriform algebra such that
\begin{align}\label{bda1}
\sigma(\id_D\otimes\prec)=(\prec\otimes\id_D)\sigma_2\sigma_1,\,
\sigma(\prec\otimes\id_D)=(\id_D\otimes\prec)\sigma_1\sigma_2,\\
\label{bda2}
\sigma(\id_D\otimes\succ)=(\succ\otimes\id_D)\sigma_2\sigma_1,\,
\sigma(\succ\otimes\id_D)=(\id_D\otimes\succ)\sigma_1\sigma_2.
\end{align}
A braided dendriform algebra $D$ is called \textit{commutative}, if $\prec\sigma=\succ$ and $\succ\sigma=\prec$.
	
For two braided dendriform algebras $(D,\prec,\succ,\sigma)$ and $(D',\prec',\succ',\sigma')$, a map $f:D\rightarrow D'$ is called a \textit{homomorphism of braided dendriform algebras}, if $f$ is a homomorphism of dendriform algebras and $(f\otimes f)\sigma=\sigma'(f\otimes f)$.
\end{defn}
For any braided dendriform algebra $(D,\prec,\succ,\sigma)$, the operator $\star:=\prec+\succ$ makes $(D,\star)$ a braided algebra, giving a splitting of $\star$ in the sense of~\cite{BBGN,Lod}.

Now we generalize \cite{Agu,EF} to the following braided situation.
\begin{prop}\label{sbrd}
Given any strongly braided Rota-Baxter algebra $(R,\mu,P,\sigma)$ of weight $\lambda$, define the operators
\[\prec_P:=\mu\left(\id_R\otimes P+\lambda\id_R^{\otimes 2}\right),\quad \succ_P:=\mu(P\otimes \id_R).\]
Then $(R,\prec_P,\succ_P,\sigma)$ is a braided dendriform algebra. If $R$ is commutative with $\lambda=0$, then $(R,\prec_P,\succ_P,\sigma)$ is even a braided commutative dendriform algebra.

\end{prop}
\begin{proof}
By~\cite{EF}, we already know that $(R,\prec_P,\succ_P)$ is a dendriform algebra. Now we only need to prove its compatibility with $\sigma$. For condition \eqref{bda1}, we have
\begin{align*}
\begin{split} \sigma&(\id_R\otimes\prec_P)=\sigma(\id_R\otimes\mu)\left(\id_R^{\otimes 2}\otimes P+\lambda\id_R^{\otimes3}\right)
=(\mu\otimes\id_R)\sigma_2\sigma_1\left(\id_R^{\otimes 2}\otimes P+\lambda\id_R^{\otimes3}\right)\\
&=(\mu\otimes\id_R)\left(\id_R\otimes P\otimes\id_R+\lambda\id_R^{\otimes3}\right)\sigma_2\sigma_1=(\prec_P\otimes\id_R)\sigma_2\sigma_1,
\end{split}\\
\begin{split} \sigma&(\prec_P\otimes\id_R)=\sigma(\mu\otimes\id_R)\left(\id_R\otimes P\otimes\id_R+\lambda\id_R^{\otimes3}\right)=(\id_R\otimes\mu)
\sigma_1\sigma_2\left(\id_R\otimes P\otimes\id_R+\lambda\id_R^{\otimes3}\right)\\
&=(\id_R\otimes\mu)\left(\id_R^{\otimes 2}\otimes P+\lambda\id_R^{\otimes3}\right)
\sigma_1\sigma_2=(\id_R\otimes\prec_P)\sigma_1\sigma_2,
\end{split}
\end{align*}
by identities \eqref{ba1} and \eqref{wbr}. The compatibility condition \eqref{bda2} between $\succ_P$ and $\sigma$ is similar to check.

When $R$ is commutative with $\lambda=0$, then $\mu\sigma=\mu$ and thus
\[\prec_P\sigma=\mu(\id_R\otimes P)\sigma
=\mu\sigma(P\otimes\id_R)=\mu(P\otimes\id_R)=\succ_P.\]	
Similarly $\succ_P\sigma=\prec_P$. Hence, $(R,\prec_P,\succ_P,\sigma)$ is a braided commutative dendriform algebra.
\end{proof}

\begin{coro}\label{bshd}
For any braided unital algebra $(A,\mu,\sigma)$ and $\lambda\in\bk$, we have the braided dendriform algebra  $(\sha_\sigma(A),\prec_P,\succ_P,\beta)$ defined by
\[\prec_P:=\diamond_{\sigma,\lambda}\left(\id_{\shap_\sigma(A)}\otimes P_{\shap_\sigma(A)}+\lambda\id_{\shap_\sigma(A)}^{\otimes 2}\right),\quad \succ_P:=\diamond_{\sigma,\lambda}(P_{\shap_\sigma(A)}
\otimes\id_{\shap_\sigma(A)}).\]
Moreover, if $A$ is commutative with $\sigma^2=\id_A^{\otimes 2}$ and $\lambda=0$, then $(\sha_\sigma(A),\prec_P,\succ_P,\beta)$ is also commutative. 	 
\end{coro}
\begin{proof}
By Proposition~\ref{wrbr}, the quadruple $(\sha_\sigma(A),\diamond_{\sigma,\lambda},P_{\shap_\sigma(A)},\beta)$ is a strongly braided Rota-Baxter algebra of weight $\lambda$. It is also commutative if $A$ is commutative with $\sigma^2=\mbox{id}_A^{\otimes 2}$ according to Theorem \ref{fbcr}. Then by Proposition~ \ref{sbrd}, it induces the desired braided dendriform algebra structure $(\sha_\sigma(A),\prec_P,\succ_P,\beta)$, which is moreover commutative under the condition $\lambda=0$.
\end{proof}

For any vector space $V$, equip the tensor space
$T(V):=\bigoplus_{i\geq0}V^{\otimes i}$ with the shuffle product $\shap$ defined recursively by
\begin{align*}
&\mathfrak{u}\,\shap 1_\bk =1_\bk\shap\,\mathfrak{u}=\mathfrak{u},\\
&\mathfrak{u}\,\shap\,\mathfrak{v} =u_1\otimes(\mathfrak{u}'\,\shap\,\mathfrak{v})
+v_1\otimes(\mathfrak{u}\,\shap\,\mathfrak{v}'),
\end{align*}
where $\mathfrak{u}=u_1\otimes\mathfrak{u}'\in V^{\otimes m},\,\mathfrak{v}=v_1\otimes\mathfrak{v}'\in V^{\otimes n}$ for $m,n>0$. Then $(T(V),\shap)$ is the usual \textit{shuffle algebra}~\cite{Kas}. Following Loday~\cite{Lod}, in the subalgebra $T^+(V)=\bigoplus_{i\geq1}V^{\otimes i}$,
one can ``split" $\shap$ into a sum $\shap=\prec_V+\succ_V$ by defining
\begin{align*}
&\mathfrak{u}\prec_V\mathfrak{v}=u_1\otimes(\mathfrak{u}'\,\shap\,\mathfrak{v}),\,
\mathfrak{u}\succ_V\mathfrak{v}=v_1\otimes(\mathfrak{u}\,\shap\,\mathfrak{v}').
\end{align*}

\begin{prop} \cite{Lod}
The triple $(T^+(V),\prec_V,\succ_V)$ is
the \textit{free commutative dendriform algebra} over $V$.
\label{freecda}
\end{prop}

For any braided vector space $(V,\sigma)$, we have the \textit{quantum shuffle algebra} $(T_\sigma(V),\shap_\sigma,\beta)$ introduced in \cite{Ros}.  Here $T_\sigma(V)$ has $T(V)$ as its underlying space, but with \textit{quantum shuffle product} $\shap_\sigma$
defined recursively as a sum $\shap_\sigma=\sum_{i,j\geq 0} \shap_{\sigma(i,j)}$ with $\shap_{\sigma(i,j)}$ sending
$V^{\otimes i}\underline{\otimes}V^{\otimes j}$ to $V^{\otimes(i+j)}$. It is just the quantum quasi-shuffle product $*_{\sigma(i,j)}$ in \eqref{eq:qqs} when $\lambda=0$. More precisely, for $m, n>0$ and $\mathfrak{u}\in V^{\otimes m},\,\mathfrak{v}\in V^{\otimes n}$, we have

\begin{align*}
&\mathfrak{u}\,\shap_{\sigma(m,0)} 1_\bk=1_\bk\shap_{\sigma(0,m)}\,\mathfrak{u}=\mathfrak{u},\\
&\mathfrak{u}\,\shap_{\sigma(m,n)}\,\mathfrak{v} :=\left(\id_V\otimes\shap_{\sigma(m-1,n)} +(\id_V\otimes\shap_{\sigma(m,n-1)})\left(\beta_{m,1}\otimes\id_V^{\otimes(n-1)}\right)\right)(\mathfrak{u}\otimes\mathfrak{v}).
\end{align*}

The quantum shuffle algebra $T_\sigma(V)$ also has a braided subalgebra on the subspace $T^+_\sigma(V)=\bigoplus_{i\geq1}V^{\otimes i}$ of $T_\sigma(V)$.
There we can split $\shap_\sigma$ into a sum $\shap_\sigma=\prec_\sigma+\succ_\sigma$ by recursively defining
\begin{align}
&\label{bd1}\mathfrak{u}\prec_\sigma\mathfrak{v} :=(\id_V\otimes\shap_{\sigma(m-1,n)})(\mathfrak{u}\otimes\mathfrak{v}),\ \mathfrak{u}\succ_\sigma\mathfrak{v} :=(\id_V\otimes\shap_{\sigma(m,n-1)})\left(\beta_{m,1}\otimes\id_V^{\otimes(n-1)}\right)(\mathfrak{u}\otimes\mathfrak{v}),
\end{align}
for $\mathfrak{u}\in V^{\otimes m},\,\mathfrak{v}\in V^{\otimes n},\,m,n>0$. Then we obtain

\begin{theorem}\label{bcd}
The quadruple $(T^+_\sigma(V),\prec_\sigma,\succ_\sigma,\beta)$ is a braided dendriform algebra, which is commutative when $\sigma^2=\id_V^{\otimes2}$.
\end{theorem}

\begin{proof}
First it is easy to check that $(T^+_\sigma(V),\prec_\sigma,\succ_\sigma)$
remains a dendriform algebra by the associativity of quantum shuffle product $\shap_\sigma$ as shown in~\cite[Theorem~2.8]{Jian1}.
So we only need to check the compatibility conditions \eqref{bda1}, \eqref{bda2}. For $\mathfrak{u}\in V^{\otimes m},\mathfrak{v}\in V^{\otimes n},\mathfrak{w}\in V^{\otimes l}$, if at least one of $m,n,l$ is $0$, the conditions are easy to check. Otherwise, suppose that $m,n,l>0$, then
\begin{align*}
\beta(\prec_\sigma&\otimes\id_{T_\sigma(V)})(\mathfrak{u}\otimes\mathfrak{v}\otimes\mathfrak{w})=\beta_{m+n,l}((\mathfrak{u}\prec_\sigma\mathfrak{v})\otimes\mathfrak{w})\\
&=T_\sigma^{\chi_{m+n,l}}\left(\id_V\otimes\sum_{w\in\mathfrak{S}_{m-1,n}}T_\sigma^w\otimes\id_V^{\otimes l}\right)(\mathfrak{u}\otimes\mathfrak{v}\otimes\mathfrak{w})\\
&=\left(\id_V^{\otimes(l+1)}\otimes\sum_{w\in\mathfrak{S}_{m-1,n}}T_\sigma^w\right)T_\sigma^{\chi_{m+n,l}}(\mathfrak{u}\otimes\mathfrak{v}\otimes\mathfrak{w})\\
&=\left(\id_V^{\otimes(l+1)}\otimes\sum_{w\in\mathfrak{S}_{m-1,n}}T_\sigma^w\right)(T_\sigma^{\chi_{m,l}}\otimes\id_V^{\otimes n})(\id_V^{\otimes m}\otimes T_\sigma^{\chi_{n,l}})(\mathfrak{u}\otimes\mathfrak{v}\otimes\mathfrak{w})\\
&=(\id_{T_\sigma(V)}\otimes\prec_\sigma)(\beta\otimes\id_{T_\sigma(V)})(\id_{T_\sigma(V)}\otimes\beta)(\mathfrak{u}\otimes\mathfrak{v}\otimes\mathfrak{w})
\end{align*}
by \eqref{beta}. On the other hand,
\begin{align*}
\beta(\id_{T_\sigma(V)}&\otimes\prec_\sigma)(\mathfrak{u}\otimes\mathfrak{v}\otimes\mathfrak{w})=\beta_{m,n+l}(\mathfrak{u}\otimes(\mathfrak{v}\prec_\sigma\mathfrak{w}))\\
&=T_\sigma^{\chi_{m,n+l}}\left(\id_V^{\otimes (m+1)}\otimes\sum_{w\in\mathfrak{S}_{n-1,l}}T_\sigma^w\right)(\mathfrak{u}\otimes\mathfrak{v}\otimes\mathfrak{w})\\
&=\left(\id_V\otimes\sum_{w\in\mathfrak{S}_{n-1,l}}T_\sigma^w\otimes\id_V^{\otimes m}\right)T_\sigma^{\chi_{m,n+l}}(\mathfrak{u}\otimes\mathfrak{v}\otimes\mathfrak{w})\\
&=\left(\id_V\otimes\sum_{w\in\mathfrak{S}_{n-1,l}}T_\sigma^w\otimes\id_V^{\otimes m}\right)(\id_V^{\otimes n}\otimes T_\sigma^{\chi_{m,l}})(T_\sigma^{\chi_{m,n}}\otimes\id_V^{\otimes l})(\mathfrak{u}\otimes\mathfrak{v}\otimes\mathfrak{w})\\
&=(\prec_\sigma\otimes\id_{T_\sigma(V)})(\id_{T_\sigma(V)}\otimes\beta)(\beta\otimes\id_{T_\sigma(V)})(\mathfrak{u}\otimes\mathfrak{v}\otimes\mathfrak{w}).
\end{align*}
This proves condition \eqref{bda1}. The verification of condition \eqref{bda2} is similar.

To finish the proof, we show that $T^+_\sigma(V)$ is commutative as a dendriform algebra when $\sigma^2=\id_V^{\otimes2}$.
Actually in this case $(T^+_\sigma(V),\shap_\sigma,\beta)$ is a braided commutative algebra and $\beta^2=\id_{T^+_\sigma(V)}^{\otimes2}$ by \cite[Lemma 14, Theorem 15]{JRZ}. For $\mathfrak{u}\in V^{\otimes m},\mathfrak{v}\in V^{\otimes n},\,m,n>0$, we have
\begin{align*}
\begin{split}
\prec_\sigma\beta&(\mathfrak{u}\otimes\mathfrak{v})
=(\id_V\otimes\shap_{\sigma(n-1,m)})\beta_{mn}
(\mathfrak{u}\otimes\mathfrak{v})\\
&=(\id_V\otimes\shap_{\sigma(m,n-1)})(\id_V\otimes\beta_{n-1,m})(\id_V\otimes\beta_{m,n-1})(\beta_{m,1}\otimes\id_V^{\otimes (n-1)})(\mathfrak{u}\otimes\mathfrak{v})\\
&=(\id_V\otimes\shap_{\sigma(m,n-1)})(\beta_{m,1}\otimes\id_V^{\otimes (n-1)})(\mathfrak{u}\otimes\mathfrak{v})=\mathfrak{u}\succ_\sigma\mathfrak{v},
\end{split}\\
\begin{split}
\succ_\sigma\beta&(\mathfrak{u}\otimes\mathfrak{v})
=(\id_V\otimes\shap_{\sigma(n,m-1)})
(\beta_{n,1}\otimes\id_V^{\otimes(m-1)})
\beta_{mn}(\mathfrak{u}\otimes\mathfrak{v})\\
&=(\id_V\otimes\shap_{\sigma(m-1,n)})(\id_V\otimes\beta_{n,m-1})
(\beta_{n,1}\otimes\id_V^{\otimes(m-1)})
\beta_{mn}(\mathfrak{u}\otimes\mathfrak{v})\\
&=(\id_V\otimes\shap_{\sigma(m-1,n)})\beta_{nm}\beta_{mn}(\mathfrak{u}\otimes\mathfrak{v})
=\mathfrak{u}\prec_\sigma\mathfrak{v},
\end{split}
\end{align*}
since $\shap_{\sigma(i,j)}=\shap_{\sigma(j,i)}\beta_{ij}$ and $\beta_{ij}=(\id_V\otimes\beta_{i,j-1})(\beta_{i,1}\otimes\id_V^{\otimes(j-1)})$.
\end{proof}

Let $\bcd$ be the category of braided commutative dendriform algebras with their homomorphisms of braided dendriform algebras as the morphisms. Consider the natural inclusion
\[j_V:V\rightarrow T^+_\sigma(V),\,v\mapsto v,\]
 then we have the following theorem as a braided generalization of Proposition~\ref{freecda}.
\begin{theorem}\label{fbcd}
	For any braided space $(V,\sigma)$ with $\sigma^2=\id_V^{\otimes 2}$, the quadruple $(T^+_\sigma(V),\prec_\sigma,\succ_\sigma,\beta)$ is the free object in $\bcd$. More precisely, for any $(D,\prec,\succ,\tau)$ in $\bcd$ with homomorphism $\psi:V\rightarrow D$ of braided vector spaces, there exists a unique morphism $\bar{\psi}:T^+_\sigma(V)\rightarrow D$ in $\bcd$ such that the following commutative diagram holds:
	\[\xymatrix@=2em{V\ar@{->}[d]_-{\psi}\ar@{->}[r]^-{j_V}&	 {T^+_\sigma(V)}\ar@{.>}[dl]^-{\bar{\psi}}\\
	D&}\]
\end{theorem}
\begin{proof}
By Theorem \ref{bcd}, $T^+_\sigma(V)$ is a braided commutative dendriform algebra when $\sigma^2=\id_A^{\otimes 2}$. It remains to verify its universal property as stated in the theorem.

Let $(D,\prec,\succ,\tau)$ be a braided commutative dendriform algebra and $\psi:V\rightarrow D$ be a homomorphism of braided vector spaces.
First note that for $\mathfrak{u}=u_1\otimes\mathfrak{u}'=u_1\otimes\cdots\otimes u_m\in V^{\otimes m}$,
we can rewrite
\begin{equation}\label{uq}
\mathfrak{u}=u_1\prec_\sigma(u_2\prec_\sigma\cdots) =u_1\prec_\sigma\mathfrak{u}'=u_1\succ_\sigma^\beta\mathfrak{u}' =\,\succ_{\sigma(m,1)}\beta_{1,m}(u_1\otimes\mathfrak{u}'),
\end{equation}
where the third equality is due to the commutativity of $T_\sigma^+(V)$, making $\succ_\sigma^\beta:=\,\succ_\sigma\beta=\prec_\sigma$. Therefore, in order to define the desired homomorphism $\bar{\psi}:T_\sigma^+(V)\rightarrow D$ of braided dendriform algebras, we must have
\begin{equation}\label{mul}
\bar{\psi}(\mathfrak{u})=\bar{\psi}(u_1\prec_\sigma\mathfrak{u}')
=\bar{\psi}(u_1)\prec\bar{\psi}(\mathfrak{u}')
=\psi(u_1)\prec\bar{\psi}(\mathfrak{u}')
=\cdots=\psi(u_1)\prec(\psi(u_2)\prec\cdots).
\end{equation}
This gives our definition of $\bar{\psi}$, and also shows its uniqueness.

Now we show that $\bar{\psi}$ is indeed a homomorphism of braided dendriform algebras. We first check $$(\bar{\psi}\otimes\bar{\psi})\beta(\mathfrak{u}\otimes\mathfrak{v})
=\tau(\bar{\psi}\otimes\bar{\psi})(\mathfrak{u}\otimes\mathfrak{v})$$ for $\mathfrak{u}\in V^{\otimes(m+1)},\,\mathfrak{v}\in V^{\otimes(n+1)},\,m,n\geq0$ by induction on $m+n$. When $m=n=0$, it is the commutativity between $\psi$ and the braidings. When exactly one of $m,n$ is 0, we use \eqref{beta}, \eqref{bda1} and \eqref{mul} to see that
\begin{align*}
(\bar{\psi}&\otimes\bar{\psi})\beta(u\otimes\mathfrak{v})
=(\bar{\psi}\otimes\psi)\beta_{1,n+1}(u\otimes\mathfrak{v})
=(\prec(\psi\otimes\bar{\psi})\otimes\psi)
\beta_{1,n+1}(u\otimes\mathfrak{v})\\
&=(\prec\otimes\id_D)
(\psi\otimes\bar{\psi}\otimes\psi)
(\id_V\otimes\beta_{1,n})
(\beta_{1,1}\otimes\id_V^{\otimes n})(u\otimes\mathfrak{v})\\
&=(\prec\otimes\id_D)\tau_2\tau_1
(\psi\otimes\psi\otimes\bar{\psi})
(u\otimes\mathfrak{v})=\tau(\id_D\otimes\prec)(\psi\otimes\psi\otimes\bar{\psi})
(u\otimes\mathfrak{v})=\tau(\bar{\psi}\otimes\bar{\psi})(u\otimes\mathfrak{v}),
\end{align*}
if $m=0,n>0$, while the case of $m>0,n=0$ is similar to check. For $m,n>0$, we have
\begin{align*}
(\bar{\psi}\otimes\bar{\psi})&\beta(\mathfrak{u}\otimes\mathfrak{v})
=(\prec(\psi\otimes\bar{\psi})\otimes\prec(\psi\otimes\bar{\psi}))\beta_{m+1,n+1}(\mathfrak{u}\otimes\mathfrak{v})\\
&=(\prec\otimes\prec)(\psi\otimes\bar{\psi}\otimes\psi\otimes\bar{\psi})(\id_V\otimes\beta_{m+1,n})(\beta_{m+1,1}\otimes\id_V^{\otimes n})(\mathfrak{u}\otimes\mathfrak{v})\\
&=(\prec\otimes\prec)(\psi\otimes\bar{\psi}\otimes\psi\otimes\bar{\psi})(\id_V\otimes\beta_{1,n}\otimes\id_V^{\otimes m})(\id_V^{\otimes 2}\otimes\beta_{m,n})(\beta_{m+1,1}\otimes\id_V^{\otimes n})(\mathfrak{u}\otimes\mathfrak{v})\\
&=(\prec\otimes\prec)\tau_2(\psi\otimes\psi\otimes\bar{\psi}\otimes\bar{\psi})(\id_V^{\otimes 2}\otimes\beta_{m,n})(\beta_{m+1,1}\otimes\id_V^{\otimes n})(\mathfrak{u}\otimes\mathfrak{v})\\
&=(\prec\otimes\prec)\tau_2\tau_3(\psi\otimes\psi\otimes\bar{\psi}\otimes\bar{\psi})(\beta_{m+1,1}\otimes\id_V^{\otimes n})(\mathfrak{u}\otimes\mathfrak{v})\\
&=(\prec\otimes\prec)\tau_2\tau_3(\psi\otimes\psi\otimes\bar{\psi}\otimes\bar{\psi})(\beta_{1,1}\otimes\id_V^{\otimes(m+n)})(\id_V\otimes\beta_{m,1}\otimes\id_V^{\otimes n})(\mathfrak{u}\otimes\mathfrak{v})\\
&=(\prec\otimes\prec)\tau_2\tau_3\tau_1(\psi\otimes\psi\otimes\bar{\psi}\otimes\bar{\psi})
(\id_V\otimes\beta_{m,1}\otimes\id_V^{\otimes n})(\mathfrak{u}\otimes\mathfrak{v})\\
&=(\prec\otimes\id_D)\tau_2(\id_D\otimes\prec\otimes\id_D)\tau_1\tau_2
(\psi\otimes\bar{\psi}\otimes\psi\otimes\bar{\psi})(\mathfrak{u}\otimes\mathfrak{v})\\
&=(\prec\otimes\id_D)\tau_2\tau_1(\prec\otimes\id_D^{\otimes 2})
(\psi\otimes\bar{\psi}\otimes\psi\otimes\bar{\psi})(\mathfrak{u}\otimes\mathfrak{v})\\
&=\tau(\id_D\otimes\prec)(\prec\otimes\id_D^{\otimes 2})
(\psi\otimes\bar{\psi}\otimes\psi\otimes\bar{\psi})(\mathfrak{u}\otimes\mathfrak{v})
=\tau(\bar{\psi}\otimes\bar{\psi})(\mathfrak{u}\otimes\mathfrak{v}).
\end{align*}
Here we have used \eqref{mul} for the first equality, \eqref{beta} for the second and third equalities, the commutativity between $\bar{\psi}$ and the braidings for the fourth to seventh equalities, also \eqref{bda1} for the eighth to tenth ones.

Finally, we check the multiplicity of $\bar{\psi}$ by induction on $m+n$. When $m,n=0$, it is the multiplicity of $\psi$. When exactly one of $m,n$ is 0, it can be obtained by \eqref{uq} and \eqref{mul}.
For $m,n>0$,
 \[\begin{split}
\bar{\psi}(\mathfrak{u})&\prec\bar{\psi}(\mathfrak{v})
=\left(\psi(u_1)\prec\bar{\psi}(\mathfrak{u}')\right)\prec\bar{\psi}(\mathfrak{v})
=\psi(u_1)\prec\left(\bar{\psi}(\mathfrak{u}')\prec\bar{\psi}(\mathfrak{v})
+\bar{\psi}(\mathfrak{u}')\succ\bar{\psi}(\mathfrak{v})\right)\\
&=\psi(u_1)\prec\left(\bar{\psi}(\mathfrak{u}'\prec_\sigma\mathfrak{v})
+\bar{\psi}(\mathfrak{u}'\succ_\sigma\mathfrak{v})\right)
=\psi(u_1)\prec\bar{\psi}(\mathfrak{u}'\shap_\sigma\mathfrak{v})\\
&=\bar{\psi}(u_1\prec_\sigma(\mathfrak{u}'\shap_\sigma\mathfrak{v}))
=\bar{\psi}(u_1\otimes(\mathfrak{u}'\shap_\sigma\mathfrak{v}))=\bar{\psi}(\mathfrak{u}\prec_\sigma\mathfrak{v}),
\end{split}\]
where the first and fifty equalities are due to \eqref{mul}, the third one is obtained by the induction hypothesis. Then
\[\begin{split}
\bar{\psi}(\mathfrak{u})\succ\bar{\psi}(\mathfrak{v})
&=\bar{\psi}(\mathfrak{u})\prec^\tau\bar{\psi}(\mathfrak{v})
=\prec\tau(\bar{\psi}\otimes\bar{\psi})(\mathfrak{u}\otimes\mathfrak{v})
=\prec(\bar{\psi}\otimes\bar{\psi})\beta(\mathfrak{u}\otimes\mathfrak{v})\\
&=\bar{\psi}\prec_\sigma\beta(\mathfrak{u}\otimes\mathfrak{v})=\bar{\psi}(\mathfrak{u}\prec_\sigma^\beta\mathfrak{v})=\bar{\psi}(\mathfrak{u}\succ_\sigma\mathfrak{v}),
\end{split}\]
where we use the commutativity of $D$ for the first equality, the commutativity between $\bar{\psi}$ and the braidings for the third one, and \eqref{uq} for the last one.
\end{proof}

\begin{remark}
Let $\svs$ be the category of braided vector spaces with symmetric braidings and $F:\bcd\rightarrow\svs$ be the forgetful functor ignoring the dendriform algebra structures. Then functor $T^+_\sigma:\svs\rightarrow\bcd$ is just the left adjoint to $F$ by Theorem \ref{fbcd}, that is,
\[\mathrm{Hom}_\svs(V,F(D))\cong\mathrm{Hom}_\bcd(T^+_\sigma(V),D).\]
\end{remark}

\section{The braided dendriform Hopf algebra of quantum shuffle}

Ronco introduced the dendriform Hopf algebra structure in~\cite{Ron1}. Here we modify it (by an opposite of $\tilde{\Delta}$) to give
\begin{defn}
A quadruple $(H,\prec,\succ,\tilde{\Delta})$ is called a \textit{dendriform Hopf algebra} if
$(H,\prec,\succ)$ is a dendriform algebra equipped with a coassociative $\bk$-linear map $\tilde{\Delta}:H\rightarrow H\otimes H$ such that
\begin{align*}
&\tilde{\Delta}(a\prec b)=\sum a'\prec b'\otimes a''\star b''+a'\prec b\otimes a''+a'\otimes a''\star b+a\prec b'\otimes b''+a\otimes b,\\
&\tilde{\Delta}(a\succ b)=\sum a'\succ b'\otimes a''\star b''+a'\succ b\otimes a''+ b'\otimes a\star b''+a\succ b'\otimes b''+b\otimes a
\end{align*}
for any $a,b\in H$, where we have used the notations $a\star b=a\prec b+a\succ b$ and $\tilde{\Delta}(a)=\sum a'\otimes a''$.
\end{defn}
In fact, for a dendriform Hopf algebra $(H,\prec,\succ,\tilde{\Delta})$, its augmented space $\bar{H}=H\oplus\bk$ becomes a bialgebra $(\bar{H},\mu,\Delta,\varepsilon)$ defined by
\begin{align*}
&1_\bk x=x1_\bk =x,\,xy=x\star y,\\
&\Delta(1_\bk)=1_\bk\otimes1_\bk,\,\Delta(x)=x\otimes 1_\bk+1_\bk\otimes x+\tilde{\Delta}(x),\\
&\varepsilon(1_\bk)=1_\bk,\,\varepsilon(x)=0,
\end{align*}
for any $x,y\in H$. In particular, if $H$ has a graded space structure compatible with $\tilde{\Delta}$, then $\bar{H}$ is a graded connected Hopf algebra.

It is natural to expect that there exists a braided version of dendriform Hopf algebras combining dendriform Hopf algebras with braided algebras.

\begin{defn}
A quintuple $(H,\prec,\succ,\tilde{\Delta},\sigma)$ is called a \textit{braided dendriform Hopf algebra} if
$(H,\prec,\succ,\sigma)$ is a braided dendriform algebra equipped with a coassociative coproduct $\tilde{\Delta}:H\rightarrow H\otimes H$ such that
\begin{align}
&\label{dha1}\tilde{\Delta}\prec=(\prec\otimes\star)\sigma_2(\tilde{\Delta}\otimes\tilde{\Delta})
+(\mbox{id}_H\otimes\star+(\prec\otimes\mbox{id}_H)\sigma_2)(\tilde{\Delta}\otimes\mbox{id}_H)+(\prec\otimes\mbox{id}_H)(\mbox{id}_H\otimes\tilde{\Delta})+\mbox{id}_H^{\otimes 2},\\
&\label{dha2}\tilde{\Delta}\succ=(\succ\otimes\star)\sigma_2(\tilde{\Delta}\otimes\tilde{\Delta})
+(\succ\otimes\mbox{id}_H)\sigma_2(\tilde{\Delta}\otimes\mbox{id}_H)+(\succ\otimes\mbox{id}_H+(\mbox{id}_H\otimes\star)\sigma_1)(\mbox{id}_H\otimes\tilde{\Delta})+\sigma.
\end{align}
\end{defn}
Consequently, the augmented space $\bar{H}=H\oplus\bk$ of a braided dendriform Hopf algebra $(H,\prec,\succ,\tilde{\Delta},\sigma)$ becomes a braided bialgebra $(\bar{H},\mu,\Delta,\varepsilon)$ with $\mu:=\,\prec +\succ$.

Define the deconcatenation map on $T_\sigma(V)$,
\[\Delta:T_\sigma(V)\rightarrow T_\sigma(V)^{\underline{\otimes} 2},
v_1\otimes\cdots\otimes v_n\mapsto\sum_{i=0}^n(v_1\otimes\cdots\otimes v_i)\,\underline{\otimes}\,(v_{i+1}\otimes\cdots\otimes v_n),\,n\geq0,\]
and its reduced one on $T^+_\sigma(V)$,
\[\bar{\Delta}:T^+_\sigma(V)\rightarrow T^+_\sigma(V)^{\underline{\otimes} 2},
v_1\otimes\cdots\otimes v_n\mapsto\sum_{i=1}^{n-1}(v_1\otimes\cdots\otimes v_i)\,\underline{\otimes}\,(v_{i+1}\otimes\cdots\otimes v_n),\,n>0,\]
with $\bar{\Delta}(v)=0$ for all $v\in V$. Also, $T_\sigma(V)$ has the counit map $\varepsilon$ such that
\[\varepsilon(1_\bk)=1_\bk\mbox{ and }\varepsilon(v_1\otimes\cdots\otimes v_n)=0,\,n>0.\]

By \cite[Proposition~1.3]{Ron2}, the free commutative dendriform algebra $(T^+(V),\prec_V,\succ_V)$ equipped with the reduced deconcatenation map $\bar{\Delta}$ is a graded dendriform Hopf algebra. In particular, the shuffle algebra $(T(V),\shap,\Delta)$ is a graded connected Hopf algebra. Now we give the braided analogue of this result.

\begin{theorem} \label{bdha}
For a braided vector space $(V,\sigma)$, the sextuple $(T_\sigma^+(V),\prec_\sigma,\succ_\sigma,\bar{\Delta},\varepsilon,\beta)$ is a braided dendriform Hopf algebra.
\end{theorem}

\begin{proof}
First by Theorem \ref{bcd}, the quadruple $(T_\sigma^+(V),\prec_\sigma,\succ_\sigma,\beta)$ is a braided dendriform algebra. We only need to check that the compatibility conditions \eqref{dha1} and \eqref{dha2} hold. It is a well-known fact that the quantum shuffle algebra $(T_\sigma(V),\shap_\sigma,\Delta,\varepsilon,\beta)$ is a braided (also graded connected) Hopf algebra \cite[\S 3]{JR}, \cite{Ros}. In particular,
\[\Delta\shap_\sigma=(\shap_\sigma\otimes\shap_\sigma)\beta_2\Delta,\]
which is due to the following decomposition of shuffle set $\mathfrak{S}_{k,n-k}$ for any fixed $0\leq k,l\leq n$,
\[\mathfrak{S}_{k,n-k}=\bigcup_{{0\leq i\leq k \atop 0\leq j\leq n-k} \atop i+j=l}(\mathfrak{S}_{ij}\times\mathfrak{S}_{k-i,n-k-j})\tau_{kn}^{ij},\]
where $\tau_{kn}^{ij}\in\mathfrak{S}_n$ is defined by
\[\tau_{kn}^{ij}(p)=
\begin{cases}
p,&\mbox{if }1\leq p\leq i\mbox{ or }k+j+1\leq p\leq n,\\
p+j,&\mbox{if }i+1\leq p\leq k,\\
p-k+i,&\mbox{if }k+1\leq p\leq k+j.\\
\end{cases}\]

On the other hand, we have the disjoint union
\[\mathfrak{S}_{k,n-k}=\mathfrak{S}^1_{k,n-k}\cup\mathfrak{S}^2_{k,n-k},\]
where $\mathfrak{S}^1_{k,n-k}=\{w\in\mathfrak{S}_{k,n-k}\,|\,w(1)=1\}$
and $\mathfrak{S}^2_{k,n-k}=\{w\in\mathfrak{S}_{k,n-k}\,|\,w(k+1)=1\}$.
Furthermore, the subsets $\mathfrak{S}^1_{k,n-k}$ and $\mathfrak{S}^2_{k,n-k}$ of $\mathfrak{S}_{k,n-k}$ also have decompositions
\[\mathfrak{S}^1_{k,n-k}=\bigcup_{{0<i\leq k \atop 0\leq j\leq n-k} \atop i+j=l}(\mathfrak{S}^1_{ij}\times\mathfrak{S}_{k-i,n-k-j})\tau_{kn}^{ij},\,
\mathfrak{S}^2_{k,n-k}=\bigcup_{{0\leq i\leq k \atop 0<j\leq n-k} \atop i+j=l}(\mathfrak{S}^2_{ij}\times\mathfrak{S}_{k-i,n-k-j})\tau_{kn}^{ij},
\]
for any fixed $0\leq k,l\leq n$. These two decomposition formulas correspond to conditions \eqref{dha1} and \eqref{dha2} for $T^+_\sigma(V)$ respectively.
\end{proof}

\section{Universal enveloping algebras of braided commutative dendriform algebras}
The functor found by Aguiar \cite{Agu} from the category of Rota-Baxter algebras to dendriform algebras has its left adjoint functor giving universal enveloping Rota-Baxter algebras of dendriform algebras. The universal enveloping Rota-Baxter algebra of a dendriform algebra was first studied in \cite{EG}. We first recall its definition.

\begin{defn}\label{uerb}
Fix $\lambda\in\bk$. For a dendriform algebra $(D,\prec,\succ)$, its \textit{universal enveloping Rota-Baxter algebra} of weight $\lambda$ is a Rota-Baxter algebra $U_{RB}(D):=U_{RB,\lambda}(D)$ with a dendriform algebra homomorphism $\rho:D\rightarrow U_{RB}(D)$ such that, for any Rota-Baxter algebra $R$ of weight $\lambda$ and dendriform algebra homomorphism $f:D\rightarrow R$, there exists a unique Rota-Baxter algebra homomorphism $\bar{f}:U_{RB}(D)\rightarrow R$ such that $f=\bar{f} \rho$.
\end{defn}

Considering all objects in Definition \ref{uerb} to the braided framework, it gives the notion of the \textit{braided universal enveloping Rota-Baxter algebra} of a braided dendriform algebra. Indeed by Proposition~\ref{sbrd}, there is a functor from the category of strongly braided Rota-Baxter algebras to braided dendriform algebras. Naturally we expect that its left adjoint functor gives universal objects associated to braided dendriform algebras in the category of strongly braided Rota-Baxter algebras.

In order to construct $U_{RB}(D)$ for a braided commutative dendriform algebra $D$, we first recall the \textit{quantum symmetric algebra}
\[S(V):=T_\sigma(V)/(\ker(\sigma+\mbox{id}_V^{\otimes 2}))\]
on a braided vector space $(V,\sigma)$. Also let $S^+(V):=T^+_\sigma(V)/(\ker(\sigma+\mbox{id}_V^{\otimes 2}))$.
Obviously the braiding $\beta$ on $T_\sigma(V)$ induces a braiding on $S(V)$, also denoted by $\beta$. Equipped with the usual concatenation product $\cdot$, $(S(V),\cdot,\beta)$ becomes a braided unital algebra by \eqref{beta}, while $(S^+(V),\cdot,\beta)$ is nonunital. If $\sigma$ is symmetric, then $(S(V),\cdot,\beta)$ is also (braided) commutative.

Furthermore, it is easy to check that the quantum symmetric algebra has the following universal property; see also \cite[Theorem 4.2]{CTS} for \textit{coboundary category} settings. For any braided commutative unital algebra $(A,\mu,\tau)$ with a braided vector space homomorphism  $f:(V,\sigma)\rightarrow(A,\tau)$, there exists a unique homomorphism $\tilde{f}:(S(V),\cdot,\beta)\rightarrow(A,\mu,\tau)$ of braided algebras such that $f=\tilde{f} i_V$, where $i_V:V\rightarrow S(V)$ is the natural inclusion. Hence, for the category of braided commutative algebras (not necessarily unital), the free object over $V$ becomes $(S^+(V),\cdot,\beta)$ if $\sigma$ is symmetric.

Now we give the following braided commutative analogue of \cite[Theorem 3.5]{EG}.
\begin{theorem}\label{berba}
For a braided commutative dendriform algebra $(D,\prec_D,\succ_D,\sigma)$ with $\sigma^2=\id_D^{\otimes 2}$, the braided universal enveloping commutative Rota-Baxter algebra $U_{RB}(D)$ is given by \[(\sha_\sigma^0(S^+(D))/J_{RB},\diamond_{\sigma,0},P_D,\beta),\]
where $J_{RB}$ is the Rota-Baxter ideal of $\sha_\sigma^0(S^+(D))$ generated by
$x\prec_D y-x\diamond_{\sigma,0}P_D(y)$
for $x,y\in D$. Here the Rota-Baxter operator of $\sha_\sigma^0(S^+(D))$, and its induced one on $\sha_\sigma^0(S^+(D))/J_{RB}$ are written as $P_D$ for short.
\end{theorem}
\begin{proof}
By Proposition \ref{fbcr1}, we know that the quadruple
\[(\sha_\sigma^0(S^+(D)),\diamond_{\sigma,0},P_D,\beta)\]
is the free braided commutative Rota-Baxter algebra of weight 0 over $S^+(D)$ in $\scr^0_\lambda$.

Since $(\sha_\sigma^0(S^+(D)),\diamond_{\sigma,0},P_D,\beta)$ is a braided commutative Rota-Baxter subalgebra of $\sha_\sigma(S(D))$,  we also know that $(\sha_\sigma^0(S^+(D)),\prec_{P_D},\succ_{P_D},\beta)$ is a braided commutative dendriform subalgebra of it by Corollary \ref{bshd}, with
\[a\prec_{P_D}b=a\diamond_{\sigma,0}P_D(b),\,a\succ_{P_D}b=P_D(a)\diamond_{\sigma,0}b,
\,a,b\in\sha_\sigma^0(S^+(D)).\]

Next we show that the braiding $\beta$ stabilizes the tensor subspace \[J_{RB}\otimes\sha_\sigma^0(S^+(D))+\sha_\sigma^0(S^+(D))\otimes J_{RB},\]
thus inducing a braiding on $\sha_\sigma^0(S^+(D))/J_{RB}$. By condition \eqref{ba1}, we only need
to prove that the images under $\beta$ of the elements
\[(x\prec_D y-x\prec_{P_D}y)\otimes a\mbox{ and }a\otimes (x\prec_D y-x\prec_{P_D}y),\]
for all $x,y\in D$ and $a\in\sha_\sigma^0(S^+(D))$ still lie in this subspace. Indeed, using conditions \eqref{bda1} and \eqref{bda2}, we have
\begin{align*}
&\begin{split}
\beta((x&\prec_D y-x\prec_{P_D}y)\otimes a)=\beta\left((\prec_D-\prec_{P_D})\otimes\mbox{id}_{\sha_\sigma^0(S^+(D))}\right)(x\otimes y\otimes a)\\&=\left(\mbox{id}_{\sha_\sigma^0(S^+(D))}\otimes(\prec_D-\prec_{P_D})\right)\beta_1\beta_2(x\otimes y\otimes a),
\end{split}\\
&\begin{split}
\beta(a\otimes&(x\prec_D y-x\prec_{P_D}y))=\beta\left(\mbox{id}_{\sha_\sigma^0(S^+(D))}\otimes(\prec_D-\prec_{P_D})\right)
(a\otimes x\otimes y)\\&=\left((\prec_D-\prec_{P_D})\otimes\mbox{id}_{\sha_\sigma^0(S^+(D))}\right)\beta_2\beta_1(a\otimes x\otimes y).
\end{split}
\end{align*}
Together with condition \eqref{wbr}, it indicates that   $J_{RB}\otimes\sha_\sigma^0(S^+(D))$ and $\sha_\sigma^0(S^+(D))\otimes J_{RB}$ are interchanged with each other under $\beta$. As a result, we have proved that $(\sha_\sigma^0(S^+(D))/J_{RB},\diamond_{\sigma,0},P_D,\beta)$ is a braided commutative Rota-Baxter (thus dendriform) quotient algebra of $\sha_\sigma^0(S^+(D))$.

Furthermore, elements
\[x\succ_D y-x\succ_{P_D}y,\,x,y\in D,\]
also lie in $J_{RB}$ by the compatibility of the dendriform algebra structures on $D$ and $\sha_\sigma^0(S^+(D))$. Consequently, we can abuse the notation and abbreviate $\prec_{P_D}$ as $\prec_D$ and $\succ_{P_D}$ as $\succ_D$ in $\sha_\sigma^0(S^+(D))/J_{RB}$.

It remains to prove the universal property. Suppose that $(R,\mu,P,\tau)$ is a braided commutative Rota-Baxter algebra such that $f:D\rightarrow R$ is a homomorphism of dendriform algebras. First by the universal property of $S^+(D)$, there exists a unique homomorphism $\tilde{f}:S^+(D)\rightarrow R$ of braided algebras such that $f=\tilde{f} i_D$.
Then by the universal property of $\sha_\sigma^0(S^+(D))$, there exists a unique homomorphism $\hat{f}:\sha_\sigma^0(S^+(D))\rightarrow R$ of braided Rota-Baxter algebras such that $\tilde{f}=\hat{f} j_{S^+(D)}$. Let $j_D=j_{S^+(D)} i_D$. Then we have $\hat{f}$ as the unique homomorphism such that $f=\hat{f} j_D$. However, $j_D$ is not a homomorphism of dendriform algebras in general.

In the end, we show that the homomorphism $\hat{f}$ can be factored through $\sha_\sigma^0(S^+(D))/J_{RB}$.
First we check
\[\begin{split}
\hat{f}(x&\prec_D y-x\diamond_{\sigma,0}P_D(y))=f(x\prec_D y)-\hat{f}(x\diamond_{\sigma,0}P_D(y))
=f(x)\prec_P f(y)-\hat{f}(x)\hat{f}(P_D(y))\\
&=f(x)\prec_P f(y)-f(x)P(f(y))=0,
\end{split}\]
for all $x,y\in D$, thus $J_{RB}\subseteq\ker\hat{f}$. It implies the existence of a homomorphism $\bar{f}:\sha_\sigma^0(S^+(D))/J_{RB}\rightarrow R$ of braided Rota-Baxter algebras such that $\hat{f}=\bar{f} \pi$ with $\pi:\sha_\sigma^0(S^+(D))\rightarrow\sha_\sigma^0(S^+(D))/J_{RB}$ as the natural projection.

Letting $\rho_D=\pi\circ j_D$, then we clearly have $\rho_D$ as a homomorphism of dendriform algebras and $\bar{f}$ as a homomorphism of braided Rota-Baxter algebras such that $f=\bar{f} \rho_D$.
To verify the uniqueness of such a homomorphism $\bar{f}$, assume that there exists another homomorphism $\bar{f}'$ satisfying $f=\bar{f}' \rho_D$, that is, $\bar{f}'\pi j_D=f=\bar{f}\pi j_D$. By the uniqueness of $\hat{f}$ with $f=\hat{f} j_D$, we have $\bar{f}'\pi=\hat{f}=\bar{f}\pi$, and thus $\bar{f}'=\bar{f}$ by the surjectivity of $\pi$.

In summary, we obtain the following commutative diagram
\[\xymatrix@=2em{&D\ar@{->}[dl]_-{i_D}\ar@{->}[ddl]_>>>>>>>{f}\ar@{->}[ddr]_>>>>>>>{\rho_D}\ar@{->}[dr]^-{j_D}&\\
	 S^+(D)\ar@{->}[rr]^-{j_{S^+(D)}}\ar@{.>}[d]_{\exists!\tilde{f}}&&
	 \sha_\sigma^0(S^+(D))\ar@{.>}[dll]_-{\exists!\hat{f}}\ar@{->}[d]_-{\pi}\\
	 R&&\sha_\sigma^0(S^+(D))/J_{RB}\,,\ar@{.>}[ll]_<<<<<<<{\exists!\bar{f}}}\]
and $(\sha_\sigma^0(S^+(D))/J_{RB},\diamond_{\sigma,0},P_D,\beta)$ is the braided universal enveloping commutative Rota-Baxter algebra of $D$.
\end{proof}

\smallskip

\noindent
{\bf Acknowledgments.}
Y. Li thanks Rutgers University -- Newark for its hospitality during his visit in 2018-2019. This work is supported by Natural Science Foundation of China (Grant Nos. 11501214, 11771142,  11771190) and the China Scholarship Council (No. 201808440068).

\bibliographystyle{amsplain}

\end{document}